\documentclass[10pt]{amsart}
\usepackage{amsmath}
\usepackage{amsxtra}
\usepackage{amscd}
\usepackage{amsthm}
\usepackage{amsfonts}
\usepackage{amssymb}
\usepackage{eucal}
\usepackage[matrix,arrow,curve]{xy}

\newtheorem{cor}[subsection]{Corollary}

\newtheorem{lem}[subsection]{Lemma}
\newtheorem{prop}[subsection]{Proposition}

\newtheorem{conj}[subsection]{Conjecture}
\newtheorem{thm}[subsection]{Theorem}

\newtheorem{rem}[subsection]{Remark}

\theoremstyle{definition}

\theoremstyle{remark}

\newcommand{\nc}{\newcommand}
\nc{\renc}{\renewcommand} \nc{\ssec}{\subsection}
\nc{\sssec}{\subsubsection} \nc{\on}{\operatorname}

\renc{\l}{\lambda}

\nc\ol{\overline} \nc\ul{\underline} \nc\wt{\widetilde}
\nc\tboxtimes{\wt{\boxtimes}} \nc{\alp}{\alpha}

\nc{\ZZ}{{\mathbb Z}} \nc{\NN}{{\mathbb N}} \nc{\CC}{{\mathbb C}}
\nc{\OO}{{\mathbb O}} \renc{\SS}{{\mathbb S}} \nc{\DD}{{\mathbb
D}}

\nc{\Fq}{{\mathbb F}_q} \nc{\Fqb}{\ol{{\mathbb F}_q}}
\nc{\Ql}{\ol{{\mathbb Q}_\ell}} \nc{\id}{\text{id}} \nc\X{\mathcal
X}

\nc{\Hom}{\on{Hom}} \nc{\Lie}{\on{Lie}} \nc{\Loc}{\on{Loc}}
\nc{\Pic}{\on{Pic}} \nc{\Bun}{\on{Bun}} \nc{\IC}{\on{IC}}
\nc{\Aut}{\on{Aut}} \nc{\rk}{\on{rk}} \nc{\Sh}{\on{Sh}}
\nc{\Perv}{\on{Perv}} \nc{\pos}{{\on{pos}}} \nc{\Conv}{\on{Conv}}
\nc{\Sph}{\on{Sph}} \nc{\Sym}{\on{Sym}}
\nc{\BunBb}{\overline{\Bun}_B} \nc{\Buno}{\overset{o}{\Bun}}
\nc{\BunPb}{{\overline{\Bun}_P}}
\nc{\BunBM}{\overline{\Bun}_{B(M)}}
\nc{\BunPbw}{{\widetilde{\Bun}_P}}
\nc{\BunBP}{\widetilde{\Bun}_{B,P}} \nc{\GUb}{\overline{G/U}}
\nc{\GUPb}{\overline{G/U(P)}}
\nc{\iso}{{\stackrel{\sim}{\longrightarrow}}}

\nc{\Hhom}{\underline{\on{Hom}}} \nc\syminfty{\on{Sym}^{\infty}}
\nc\lal{\ol{\lambda}} \nc\xl{\ol{x}} \nc\thl{\ol{\theta}}
\nc\nul{\ol{\nu}} \nc\mul{\ol{\mu}} \nc\Sum\Sigma
\nc{\oX}{\overset{o}{X}{}}

\nc{\M}{{\mathcal M}} \nc{\N}{{\mathcal N}} \nc{\F}{{\mathcal F}}
\nc{\D}{{\mathcal D}} \nc{\Q}{{\mathcal Q}} \nc{\Y}{{\mathcal Y}}
\nc{\G}{{\mathcal G}} \nc{\E}{{\mathcal E}} \nc{\CalC}{{\mathcal
C}}
\nc\Dh{\widehat{\D}}

\nc{\C}{{\mathcal C}} \nc{\K}{{\mathcal K}}
\renewcommand{\H}{{\mathcal H}}

\nc{\T}{{\mathcal T}} \nc{\V}{{\mathcal V}} \renc{\P}{{\mathcal
P}} \nc{\A}{{\mathcal A}} \nc{\B}{{\mathcal B}} \nc{\U}{{\mathcal
U}}

\nc{\Gr}{\on{Gr}}

\nc{\frn}{{\check{\mathfrak u}(P)}}
\nc\f{{\mathfrak f}}

\nc{\q}{{\mathfrak q}} \nc{\p}{{\mathfrak p}} \nc{\s}{{\mathfrak
s}} \nc\w{\text{w}}

\nc\mathi\iota \nc\Spec{\on{Spec}} \nc\Mod{\on{Mod}}
\nc{\tw}{\widetilde{\mathfrak t}} \nc{\pw}{\widetilde{\mathfrak
p}} \nc{\qw}{\widetilde{\mathfrak q}} \nc{\jw}{\widetilde j}

\nc{\grb}{\overline{\Gr}} \nc{\I}{\mathcal I}

\nc{\lambdach}{{\check\lambda}} \nc{\Lambdach}{{\check\Lambda}{}}
\nc{\much}{{\check\mu}} \nc{\omegach}{{\check\omega}}
\nc{\nuch}{{\check\nu}} \nc{\etach}{{\check\eta}}
\nc{\alphach}{{\check\alpha}} \nc{\betach}{{\check\beta}}
\nc{\rhoch}{{\check\rho}} \nc{\ch}{{\check h}}

\nc{\Hb}{\overline{\H}}

\nc{\BA}{{\mathbb{A}}} \nc{\BC}{{\mathbb{C}}} \nc{\BQ}{{\mathbb{Q}}}
\nc{\BM}{{\mathbb{M}}} \nc{\BN}{{\mathbb{N}}}
\nc{\BP}{{\mathbb{P}}} \nc{\BR}{{\mathbb{R}}}
\nc{\BZ}{{\mathbb{Z}}} \nc{\BS}{{\mathbb{S}}} \nc{\BV}{{\mathbb{V}}}

\nc{\CA}{{\mathcal{A}}} \nc{\CB}{{\mathcal{B}}}
\nc{\CE}{{\mathcal{E}}} \nc{\CF}{{\mathcal{F}}}
\nc{\CG}{{\mathcal{G}}} \nc{\CH}{{\mathcal{H}}}
\nc{\CI}{{\mathcal{I}}} \nc{\CL}{{\mathcal{L}}}
\nc{\CM}{{\mathcal{M}}} \nc{\CN}{{\mathcal{N}}}
\nc{\CO}{{\mathcal{O}}} \nc{\CP}{{\mathcal{P}}}
\nc{\CQ}{{\mathcal{Q}}} \nc{\CR}{{\mathcal{R}}}
\nc{\CS}{{\mathcal{S}}} \nc{\CT}{{\mathcal{T}}}
\nc{\CU}{{\mathcal{U}}} \nc{\CV}{{\mathcal{V}}}  \nc{\CY}{{\mathcal Y}}
\nc{\CW}{{\mathcal{W}}} \nc{\CZ}{{\mathcal{Z}}}

\nc{\cM}{{\check{\mathcal M}}{}} \nc{\csM}{{\check{\mathcal A}}{}}
\nc{\oM}{{\overset{\circ}{\mathcal M}}{}}
\nc{\obM}{{\overset{\circ}{\mathbf M}}{}}
\nc{\oCA}{{\overset{\circ}{\mathcal A}}{}}
\nc{\obA}{{\overset{\circ}{\mathbf A}}{}}
\nc{\ooM}{{\overset{\circ}{M}}{}}
\nc{\osM}{{\overset{\circ}{\mathsf M}}{}}
\nc{\vM}{{\overset{\bullet}{\mathcal M}}{}}
\nc{\nM}{{\underset{\bullet}{\mathcal M}}{}}
\nc{\oD}{{\overset{\circ}{\mathcal D}}{}}
\nc{\obD}{{\overset{\circ}{\mathbf D}}{}}
\nc{\oA}{{\overset{\circ}{\mathbb A}}{}}
\nc{\op}{{\overset{\bullet}{\mathbf p}}{}}
\nc{\cp}{{\overset{\circ}{\mathbf p}}{}}
\nc{\oU}{{\overset{\bullet}{\mathcal U}}{}}
\nc{\oZ}{{\overset{\circ}{\mathcal Z}}{}}
\nc{\ofZ}{{\overset{\circ}{\mathfrak Z}}{}}

\nc{\ff}{{\mathfrak{f}}} \nc{\fv}{{\mathfrak{v}}}
\nc{\fa}{{\mathfrak{a}}} \nc{\fb}{{\mathfrak{b}}}
\nc{\fd}{{\mathfrak{d}}} \nc{\fe}{{\mathfrak{e}}}
\nc{\fg}{{\mathfrak{g}}} \nc{\fgl}{{\mathfrak{gl}}}
\nc{\fh}{{\mathfrak{h}}} \nc{\fri}{{\mathfrak{i}}}
\nc{\fj}{{\mathfrak{j}}} \nc{\fk}{{\mathfrak{k}}}
\nc{\fm}{{\mathfrak{m}}} \nc{\fn}{{\mathfrak{n}}}
\nc{\ft}{{\mathfrak{t}}} \nc{\fu}{{\mathfrak{u}}}
\nc{\fw}{{\mathfrak{w}}} \nc{\fz}{{\mathfrak{z}}}
\nc{\fp}{{\mathfrak{p}}} \nc{\frr}{{\mathfrak{r}}}
\nc{\fs}{{\mathfrak{s}}} \nc{\fsl}{{\mathfrak{sl}}}
\nc{\hsl}{{\widehat{\mathfrak{sl}}}}
\nc{\hgl}{{\widehat{\mathfrak{gl}}}}
\nc{\hg}{{\widehat{\mathfrak{g}}}}
\nc{\chg}{{\widehat{\mathfrak{g}}}{}^\vee}
\nc{\hn}{{\widehat{\mathfrak{n}}}}
\nc{\chn}{{\widehat{\mathfrak{n}}}{}^\vee}

\nc{\fA}{{\mathfrak{A}}} \nc{\fB}{{\mathfrak{B}}}
\nc{\fD}{{\mathfrak{D}}} \nc{\fE}{{\mathfrak{E}}}
\nc{\fF}{{\mathfrak{F}}} \nc{\fG}{{\mathfrak{G}}} \nc{\fH}{{\mathfrak{H}}}
\nc{\fI}{{\mathfrak{I}}} \nc{\fJ}{{\mathfrak{J}}}
\nc{\fK}{{\mathfrak{K}}} \nc{\fL}{{\mathfrak{L}}}
\nc{\fM}{{\mathfrak{M}}} \nc{\fN}{{\mathfrak{N}}}
\nc{\frP}{{\mathfrak{P}}} \nc{\fQ}{{\mathfrak{Q}}}
\nc{\fT}{{\mathfrak{T}}} \nc{\fU}{{\mathfrak{U}}}
\nc{\fV}{{\mathfrak{V}}} \nc{\fW}{{\mathfrak{W}}}
\nc{\fX}{{\mathfrak{X}}} \nc{\fY}{{\mathfrak{Y}}}
\nc{\fZ}{{\mathfrak{Z}}}

\nc{\ba}{{\mathbf{a}}}
\nc{\bb}{{\mathbf{b}}} \nc{\bc}{{\mathbf{c}}}
\nc{\be}{{\mathbf{e}}} \nc{\bj}{{\mathbf{j}}}
\nc{\bn}{{\mathbf{n}}} \nc{\bp}{{\mathbf{p}}}
\nc{\bq}{{\mathbf{q}}} \nc{\br}{{\mathbf{r}}} \nc{\bt}{{\mathbf{t}}}
\nc{\bfu}{{\mathbf{u}}} \nc{\bv}{{\mathbf{v}}}
\nc{\bx}{{\mathbf{x}}} \nc{\by}{{\mathbf{y}}}
\nc{\bw}{{\mathbf{w}}} \nc{\bA}{{\mathbf{A}}}
\nc{\bB}{{\mathbf{B}}} \nc{\bC}{{\mathbf{C}}}
\nc{\bD}{{\mathbf{D}}} \nc{\bF}{{\mathbf{F}}}
\nc{\bH}{{\mathbf{H}}} \nc{\bK}{{\mathbf{K}}}
\nc{\bM}{{\mathbf{M}}} \nc{\bN}{{\mathbf{N}}}
\nc{\bO}{{\mathbf{O}}} \nc{\bS}{{\mathbf{S}}} \nc{\bT}{{\mathbf{T}}}
\nc{\bV}{{\mathbf{V}}} \nc{\bW}{{\mathbf{W}}}
\nc{\bX}{{\mathbf{X}}}
\nc{\bY}{{\mathbf{Y}}} \nc{\bP}{{\mathbf{P}}}
\nc{\bZ}{{\mathbf{Z}}} \nc{\bh}{{\mathbf{h}}}

\nc{\sA}{{\mathsf{A}}} \nc{\sB}{{\mathsf{B}}}
\nc{\sC}{{\mathsf{C}}} \nc{\sD}{{\mathsf{D}}}
\nc{\sE}{{\mathsf{E}}} \nc{\sF}{{\mathsf{F}}}
\nc{\sK}{{\mathsf{K}}} \nc{\sL}{{\mathsf{L}}}
\nc{\sM}{{\mathsf{M}}} \nc{\sO}{{\mathsf{O}}}
\nc{\sQ}{{\mathsf{Q}}} \nc{\sP}{{\mathsf{P}}}
\nc{\sT}{{\mathsf{T}}} \nc{\sZ}{{\mathsf{Z}}}
\nc{\sV}{{\mathsf{V}}}
\nc{\sfp}{{\mathsf{p}}} \nc{\sr}{{\mathsf{r}}}
\nc{\st}{{\mathsf{t}}} \nc{\sfb}{{\mathsf{b}}}
\nc{\sfc}{{\mathsf{c}}} \nc{\sd}{{\mathsf{d}}}
\nc{\sz}{{\mathsf{z}}}

\nc{\BK}{{\bar{K}}}

\nc{\tA}{{\widetilde{\mathbf{A}}}}
\nc{\tB}{{\widetilde{\mathcal{B}}}}
\nc{\tg}{{\widetilde{\mathfrak{g}}}} \nc{\tG}{{\widetilde{G}}}
\nc{\TM}{{\widetilde{\mathbb{M}}}{}}
\nc{\tO}{{\widetilde{\mathsf{O}}}{}}
\nc{\tU}{{\widetilde{\mathfrak{U}}}{}} \nc{\TZ}{{\tilde{Z}}}
\nc{\tx}{{\tilde{x}}} \nc{\tbv}{{\tilde{\bv}}}
\nc{\tfP}{{\widetilde{\mathfrak{P}}}{}} \nc{\tz}{{\tilde{\zeta}}}
\nc{\tmu}{{\tilde{\mu}}}

\nc{\urho}{\underline{\rho}} \nc{\uB}{\underline{B}}
\nc{\uC}{{\underline{\mathbb{C}}}} \nc{\ui}{\underline{i}}
\nc{\uj}{\underline{j}} \nc{\ofP}{{\overline{\mathfrak{P}}}}
\nc{\oB}{{\overline{\mathcal{B}}}}
\nc{\og}{{\overline{\mathfrak{g}}}} \nc{\oI}{{\overline{I}}}

\nc{\eps}{\varepsilon} \nc{\hrho}{{\hat{\rho}}}
\nc{\blambda}{{\boldsymbol{\lambda}}}

\nc{\one}{{\mathbf{1}}} \nc{\two}{{\mathbf{t}}}

\nc{\Rep}{{\mathop{\operatorname{\rm Rep}}}}
\nc{\Tot}{{\mathop{\operatorname{\rm Tot}}}}
\nc{\Ker}{{\mathop{\operatorname{\rm Ker}}}}
\nc{\Hilb}{{\mathop{\operatorname{\rm Hilb}}}}
\nc{\End}{{\mathop{\operatorname{\rm End}}}}
\nc{\Ext}{{\mathop{\operatorname{\rm Ext}}}}
\nc{\CHom}{{\mathop{\operatorname{{\mathcal{H}}\it om}}}}
\nc{\GL}{{\mathop{\operatorname{\rm GL}}}}
\nc{\gr}{{\mathop{\operatorname{\rm gr}}}}
\nc{\Id}{{\mathop{\operatorname{\rm Id}}}}
\nc{\defi}{{\mathop{\operatorname{\rm def}}}}
\nc{\length}{{\mathop{\operatorname{\rm length}}}}
\nc{\supp}{{\mathop{\operatorname{\rm supp}}}}

\nc{\Cliff}{{\mathsf{Cliff}}}
\nc{\Fl}{{\mathsf{Fl}}} \nc{\Fib}{{\mathsf{Fib}}}
\nc{\Coh}{{\mathsf{Coh}}} \nc{\FCoh}{{\mathsf{FCoh}}}

\nc{\reg}{{\text{\rm reg}}}

\nc{\cplus}{{\mathbf{C}_+}} \nc{\cminus}{{\mathbf{C}_-}}
\nc{\cthree}{{\mathbf{C}_*}} \nc{\Qbar}{{\bar{Q}}}

\nc{\bOmega}{{\overline{\Omega}}}

\nc{\seq}[1]{\stackrel{#1}{\sim}}

\nc{\aff}{\operatorname{aff}}

\nc{\ep}{\varepsilon}

\title{Cactus group and monodromy of Bethe vectors}
\dedicatory{To my wife Sasha}
\author{Leonid Rybnikov}
\email{leo.rybnikov@gmail.com}
\address{National Research University Higher School of Economics,
Department of Mathematics,
International Laboratory of Representation Theory and Mathematical Physics,
and Institute for Information Transmission Problems,
20 Myasnitskaya st,
Moscow 101000, Russia}

\begin{document}
\maketitle

\begin{abstract} Cactus group is the fundamental group of the real locus of the Deligne-Mumford moduli space of stable rational curves. This group appears naturally as an analog of the braid group in coboundary monoidal categories. We define an action of the cactus group on the set of Bethe vectors of the Gaudin magnet chain corresponding to arbitrary semisimple Lie algebra $\fg$. Cactus group appears in our construction as a subgroup in the Galois group of Bethe Ansatz equations. Following the idea of Pavel Etingof, we conjecture that this action is isomorphic to the action of the cactus group on the tensor product of crystals coming from the general coboundary category formalism. We prove this conjecture in the case $\fg=\fsl_2$ (in fact, for this case the conjecture almost immediately follows from the results of Varchenko on asymptotic solutions of the KZ equation and crystal bases). We also present some conjectures generalizing this result to Bethe vectors of shift of argument subalgebras and relating the cactus group with the Berenstein-Kirillov group of piecewise-linear symmetries of the Gelfand-Tsetlin polytope.
\end{abstract}

\section{Introduction}

\subsection{Gaudin algebras.} The Gaudin model was introduced in \cite{G1} as a spin model related
to the Lie algebra $\fsl_2$, and generalized to the case of arbitrary
semisimple Lie algebras in \cite{G}, 13.2.2. The generalized Gaudin
model has the following algebraic interpretation.

Let $\{x_a\},\ a=1,\dots,\dim\fg$, be an orthonormal basis of $\fg$ with respect to the standard invariant inner product. For any $x\in U(\fg)$, consider the element $x^{(i)}=1\otimes\dots\otimes 1\otimes
x\otimes 1\otimes\dots\otimes 1\in U(\fg)^{\otimes n}$ ($x$ stands on the $i$th place). Let $V_{\lambda}$ be an irreducible representation of a semisimple (reductive) Lie algebra $\fg$ with the highest weight
$\lambda$. For any collection of integral dominant weights $(\lambda)=\lambda_1,\dots,\lambda_n$, let
$\BV_{\ul{\l}}=V_{\l_1}\otimes\dots\otimes V_{\l_n}$. We fix a collection $\ul{z}:=(z_1,z_2,\ldots,z_n)$ of pairwise distinct complex numbers. The Hamiltonians of Gaudin model are the following commuting operators
acting in the space $\BV_{\ul{\l}}$:
\begin{equation}\label{quadratic}
H_i=\sum\limits_{j\neq i}\sum\limits_{a=1}^{\dim\fg}
\frac{x_a^{(i)}x_a^{(j)}}{z_i-z_j}.
\end{equation}

We can treat the $H_i$ as elements of the universal enveloping
algebra $U(\fg)^{\otimes n}$. In \cite{FFR}, the existence of  a
large commutative subalgebra $\A(\ul{z})=\A(z_1,\dots,z_n)\subset
U(\fg)^{\otimes n}$
containing $H_i$ was proved. This subalgebra commutes with the diagonal action of $\fg$ on $U(\fg)^{\otimes n}$ and in fact it is a maximal commutative subalgebra in $[U(\fg)^{\otimes n}]^{\fg}$.

For $\fg=\fsl_2$, the subalgebra $\A(\ul{z})\subset
U(\fg)^{\otimes n}$ is generated by the elements $H_i$ and the center of $U(\fg)^{\otimes n}$. In other cases, the algebra
$\A(\ul{z})$ has also some new generators known as higher
Gaudin Hamiltonians. This algebra is known to be a polynomial algebra with $\frac{n-1}{2}\dim\fg+\frac{n+1}{2}\rk\fg$ generators. We will call $\A(\ul{z})$ the \emph{Gaudin algebra}.

\subsection{Bethe Ansatz conjecture.} The main problem in Gaudin model is the problem of simultaneous
diagonalization of (higher) Gaudin Hamiltonians. It follows from the
\cite{FFR} construction that all elements of
$\A(\ul{z})\subset U(\fg)^{\otimes n}$ are invariant with
respect to the diagonal action of $\fg$, and therefore it is
sufficient to diagonalize the algebra $\A(\ul{z})$ in the
subspace $\BV_{\ul{\l}}^{sing}\subset \BV_{\ul{\l}}$ of singular vectors
with respect to the diagonal
action of $\fg$. In many important cases, the Gaudin eigenproblem is solved by the \emph{algebraic Bethe Ansatz} method which provides an explicit (but complicated) construction of joint eigenvectors for $\A(\ul{z})$ in $\BV_{\ul{\l}}^{sing}$, see \cite{FFR} for more details. The famous Bethe Ansatz conjecture states that this method always works, i.e. gives an eigenbasis for $\A(\ul{z})$ in $\BV_{\ul{\l}}^{sing}$.
In particular, the conjecture says that, for generic $\ul{z}$, the algebra $\A(\ul{z})$ has simple spectrum in
$\BV_{\ul{\l}}^{sing}$. The latter was proved in \cite{MTV07} for $\fg=\fsl_N$. More precisely, it is proved that the space $\BV_{\ul{\l}}^{sing}$ is always cyclic as $\A(\ul{z})$-module, and hence $\A(\ul{z})$ has simple spectrum whenever acts by semisimple operators. On the other hand, for \emph{real} values of the parameters $z_i$, the algebra $\A(\ul{z})$ is generated by Hermitian (hence semisimple) operators, hence has simple spectrum.

Generally, Bethe eigenvectors (and the corresponding eigenvalues) are not rational functions of the $z_i$'s, and hence there is a nontrivial Galois group action on Bethe eigenvectors. Our first motivation for the present work is to understand this Galois group action.

\subsection{Closure of the family $\A(\ul{z})$.} The family $\A(\ul{z})$, as defined, is parameterized by a noncompact complex algebraic variety of configurations of pairwise distinct points on the complex line. On the other hand, every subalgebra is (in appropriate sense) a point of some Grassmann variety which is compact. Hence there is a family of commutative subalgebras which extends the family $\A(\ul{z})$ and is parameterized by some compact variety. Our second motivation for the present work is to understand this compactification. According to Aguirre, Felder and Veselov \cite{AFV}, the closure of the family of quadratic Gaudin Hamiltonians is parameterized by the Deligne-Mumford compactification $\overline{M_{0,n+1}}$ of the moduli space of stable rational curves with $n+1$ marked points. We prove that the closure of the family $\A(\ul{z})$ is also parameterized by $\overline{M_{0,n+1}}$ (i.e. there are no additional blow-ups). Furthermore, we prove that the natural topological operad structure on $\overline{M_{0,n+1}}$ is compatible with that on commutative subalgebras of $U(\fg)^{\otimes n}$. This allows to describe explicitly the algebras corresponding to boundary points of $\overline{M_{0,n+1}}$ and to prove that they always have a cyclic vector in $\BV_{\ul{\l}}^{sing}$. We deduce from this the simple spectrum property for the subalgebras attached to \emph{all real points} of $\overline{M_{0,n+1}}$.

This allows us to regard the eigenbasis (or, more precisely, the set of $1$-dimensional eigenspaces) of $\A(\ul{z})$ in $\BV_{\ul{\l}}^{sing}$ as a \emph{covering} of the space $\overline{M_{0,n+1}}(\BR)$. Denote the fiber of this covering at a point $\ul{z}\in\overline{M_{0,n+1}}(\BR)$ by $B_{\ul{\l}}(\ul{z})$. The fundamental group of $\overline{M_{0,n+1}}(\BR)$ (called \emph{pure cactus group} $PJ_n$) acts on this set. This gives a homomorphism from $PJ_n$ to the Galois group of Bethe eigenvalues.

\begin{rem} Generally, this Galois group is bigger than the image of $PJ_n$. The smallest example in which this occurs is $\fg=\fsl_2$, $n=3$, $\l_1=\l_2=\l_3=2$: since $\overline{M_{0,4}}(\BR)=\BR\BP^1$ we have $PJ_3=\BZ$ and hence its image is commutative. On the other hand, the Galois group is $S_3$ (this was recently shown by Azad Saifullin \cite{Sai}).
\end{rem}

\subsection{Cactus group.} The group $PJ_n:=\pi_1(\overline{M_{0,n+1}}(\BR))$  can be described as follows. Let $J_n$ be the group with the generators $s_{p,q},\ 1\le p<q\le n$, and the defining relations
\begin{equation*}
\begin{array}{l}s_{p,q}^2=e;\\
s_{p_1,q_1}s_{p_2,q_2}=s_{p_2,q_2}s_{p_1,q_1}\ \text{if}\ q_1<p_2;\\
s_{p_1,q_1}s_{p_2,q_2}s_{p_1,q_1}=s_{p_1+q_1-q_2,p_1+q_1-p_2}\ \text{if}\ p_1\le p_2<q_2\le q_1.
\end{array}
\end{equation*}
There is an epimorphism $\pi:J_n\to S_n$ which takes $s_{p,q}$ to the involution reversing the segment $\{p,\ldots,q\}\subset\{1,\ldots,n\}$. According to \cite{DJS,D}, $J_n$ is the orbifold fundamental group of $\overline{M_{0,n+1}}(\BR)/S_n$, and hence $PJ_n\simeq \Ker\pi$. In \cite{HK} the groups $J_n$ and $PJ_n$ were named \emph{cactus group} and \emph{pure cactus group}, respectively.

It was observed by Henriques and Kamnitzer in \cite{HK} that the groups $J_n$ and $PJ_n$ naturally arise in \emph{coboundary categories}. That is, monoidal category with a functorial involutive isomorphism $s_{X,Y}:X\otimes Y\to Y\otimes X$, called \emph{commutor}, satisfying certain natural relations. Coboundary category is an analog of braided monoidal category where the role of the braid group $B_n$ is played by the cactus group $J_n$. In particular, the pure cactus group $PJ_n$ acts by endomorphisms of $X_1\otimes\ldots\otimes X_n$ for any collection of objects of any coboundary category.

The main example of a coboundary category is the category of $\fg$-\emph{crystals} for a Kac-Moody algebra $\fg$. Loosely speaking, $\fg$-crystal is the $q\to\infty$ limit of a $U_q(\fg)$-module. In this limit, $U_q(\fg)$-modules are replaced by colored oriented graphs with the vertices representing the basis vectors and the edges representing the action of the Chevalley generators of $U_q(\fg)$. There is a well-defined tensor product on $\fg$-crystals which is \emph{not} symmetric, but tensor products of the same objects in different order are still isomorphic. The commutor is a functorial choice of such isomorphism satisfying some natural axioms. The commutor for the tensor product of crystals for finite-dimensional $\fg$ was first defined by Henriques and Kamnitzer in \cite{HK} in a purely combinatorial way. Later in \cite{KT} Kamnitzer and Tingley gave an equivalent definition in terms of the unitarized $R$-matrix. For general Kac-Moody algebra, the crystal commutor was defined by Savage in \cite{S}.

Consider the tensor product $\CB_{\l_1}\otimes\ldots\otimes\CB_{\l_n}$ of the $\fg$-crystals with highest weights $\l_1,\ldots,\l_n$. The commutor gives an action of the pure cactus group on the set $\CB_{\ul{\l}}$ of highest elements of this tensor product. Note that $\CB_{\ul{\l}}$ has the same cardinality as $B_{\ul{\l}}(\ul{z})$.

\begin{conj}\label{conj-etingof} \emph{(Pavel Etingof)} The actions of $PJ_n$ on $B_{\ul{\l}}(\ul{z})$ and on $\CB_{\ul{\l}}$ are isomorphic.
\end{conj}

We prove this conjecture for $\fg=\fsl_2$ in two different ways. The first way, suggested by Pavel Etingof, is to use the Drinfeld-Kohno theorem in its ``crystal'' limit $q\to\infty$. This relates the Gaudin model (on the KZ side) with the crystal (on the quantum group side). In fact, all necessary ingredients for this are already contained in the papers of Varchenko \cite{Var} and Kamnitzer--Tingley \cite{KT}. The second way is to relate the monodromy of Bethe vectors with the ``hive'' realization of the category of crystals from \cite{HK2}. For $\fg=\fsl_2$, the eigenvectors of $\A(\ul{z})$ at the vertices (i.e. $0$-dimensional strata) of $\ol{M_{0,n+1}}$ are indexed by integer points of a convex polytope depending on $\ul{\l}$ and on the vertex. The transports along $1$-dimensional strata of $\ol{M_{0,n+1}}(\BR)$ give some natural bijections between the sets of integer points of the polytopes at different vertices of $\ol{M_{0,n+1}}$. We show that these bijections come from piecewise linear transformations of the corresponding polytopes, and relate them to the octahedron recurrence. This gives another (purely combinatorial) proof of Conjecture~\ref{conj-etingof}.

\subsection{The paper is organized as follows.} In section~\ref{sect-deligne-mumford} we recall some basic definitions and well-known facts regarding the Deligne-Mumford compactification $\ol{M_{0,n+1}}$. In section~\ref{sect-bethe-alg} we summarize the known facts about the family of Bethe algebras $\A(\ul{z})$ and prove our first main result that the closure of this family is parameterized by $\ol{M_{0,n+1}}$ and that for every algebra from the closure the module $\BV_{\ul{\l}}^{sing}$ is cyclic. In sections~\ref{sect-kamnitzer}~and~\ref{sect-varchenko} we summarize the necessary ingredients (from \cite{KT} and \cite{Var}, respectively) for the proof of Etingof's conjecture. In section~\ref{sect-main} we prove Etingof's conjecture for $\fg=\fsl_2$ (this is our second main result). In section~\ref{sect-piecelinear} we describe the piecewise linear transformations of the polytopes arising from our construction and give a combinatorial proof of Etingof's conjecture. Section~\ref{sect-conjectures} is devoted to conjectures generalizing our results.

\subsection{Acknowledgements.} I am grateful to Pavel Etingof for stating Conjecture~\ref{conj-etingof} and for explaining to me the unitarization construction of the crystal commutor from \cite{KT}. I am happy to thank Arkady Berenstein, Nick Early, Alexander Goncharov, Joel Kamnitzer, Evgeny Mukhin and Vitaly Tarasov for extremely useful discussions and references.

The author was supported by the Russian President grant MK-2121.2014.1. The latest version of the article was prepared within the framework of the Academic Fund Program at the National Research University Higher School of Economics (HSE) in 2015- 2016 (grant №15-01-0062)  and supported within the framework of a subsidy granted to the HSE by the Government of the Russian Federation for the implementation of the Global Competitiveness Program.

\section{The space $\overline{M_{0,n+1}}$}\label{sect-deligne-mumford}

\subsection{} Let $\overline{M_{0,n+1}}$ denote the Deligne-Mumford space of stable rational curves with $n+1$ marked points. The points of $\overline{M_{0,n+1}}$ are isomorphism classes of curves of genus $0$, with $n+1$ ordered marked points and possibly with nodes, such that each component has at least $3$ distinguished points (either marked points or nodes). One can represent the combinatorial type of such a curve as a tree with $n+1$ leaves with inner vertices representing irreducible components of the corresponding curve, inner edges corresponding to the nodes and the leaves corresponding to the marked points. Informally, the topology of $\overline{M_{0,n+1}}$ is determined by the following rule: when some of the distinguished points (marked or nodes) from the same component collide, they bubble off into a new component.

The space $\overline{M_{0,n+1}}$ is a smooth algebraic variety. It can be regarded as a compactification of the configuration space $M_{0,n+1}$ of ordered $(n+1)$-tuples $(z_1,z_2,\ldots,z_{n+1})$ of pairwise distinct points on $\BC\BP^1$ modulo the automorphism group $PGL_2(\BC)$. Since the group $PGL_2(\BC)$ acts transitively on triples of distinct points, we can fix the $(n+1)$-th point to be $\infty\in\BC\BP^1$ and fix the sum of coordinates of other points to be zero. Then the space $M_{0,n+1}$ gets identified with the quotient ${\rm Conf}_n / \BC^*$ where ${\rm Conf}_n:=\{(z_1,\ldots,z_n)\in\BC^n\ |\ z_i\ne z_j,\ \sum\limits_{i=1}^nz_i=0\}$, and the group $\BC^*$ acts by dilations. Under this identification of $M_{0,n+1}$, the space $\overline{M_{0,n+1}}$ is just the GIT quotient by $\BC^*$ of the iterated blow-up of the subspaces of the form $\{z_{i_1}=z_{i_2}=\ldots=z_{i_k}\}$ in $\BC^{n-1}$. The space $\overline{M_{0,n+1}}$ comes with the tautological bundles $\CL_i$ whose fiber is the line representing the point $z_i$. The total space $\widetilde{M_{0,n+1}}$ of the tautological line bundle $\CL_{n+1}$ is then just the blow-up, without taking the quotient. ${\rm Conf}_n$ is a Zariski open subset in $\widetilde{M_{0,n+1}}$.

The space $\overline{M_{0,n+1}}$ is stratified as follows. The strata are indexed by the combinatorial types of stable rational curves, i.e. by rooted trees with $n$ leaves colored by the marked points $z_1,\ldots,z_n$ (the root is colored by $z_{n+1}=\infty$). Let $T$ be such a tree, then the corresponding stratum $M_T$ is the product of $M_{0,k(I)}$ over all inner vertices $I$ of $T$ with $k(I)$ being the index of $I$. In particular, $0$-dimensional strata correspond to binary rooted trees with $n$ (ordered) leaves. The stratum corresponding to a tree $T$ lies in the closure of the one corresponding to a tree $T'$ if and only if $T'$ is obtained from $T$ by contracting some edges.

\subsection{Operad structure on $\ol{M_{0,n+1}}$.} The spaces $\ol{M_{0,n+1}}$ form a topological operad. This means that one can regard each point of the space $\ol{M_{0,n+1}}$ as an $n$-ary operation with the inputs at marked points $z_1,\ldots,z_n$ and the output at $z_{n+1}$. Then one can substitute any operation of this form to each of the inputs. More precisely, for any partition of the set $\{1,\ldots,n\}$ into the disjoint union of subsets $M_1,\ldots,M_k$ with $|M_i|=m_i\ge1$ there is a natural substitution map $\gamma_{k;M_1,\ldots,M_k}:\ol{M_{0,k+1}}\times\prod\limits_{i=1}^k\ol{M_{0,m_i+1}}\to\ol{M_{0,n+1}}$ which attaches the $i$-th curve $C_i\in\ol{M_{0,m_i+1}}$ to the $i$-th marked point of the curve $C_0\in\ol{M_{0,k+1}}$ by gluing the $m_{i}+1$-th marked point of each $C_i$ with the $i$-th marked point of $C_0$. One can extend the definition of $\ol{M_{0,n+1}}$ and take $\ol{M_{0,2}}=pt$ (defining the (unique) curve $C\in\ol{M_{0,2}}$ also to be a point). Then the substitution maps with $m_i=1$ are still well-defined and moreover all substitution maps $\gamma_{k;M_1,\ldots,M_k}$ are compositions of the elementary ones with $m_1=\ldots=m_{k-1}=1$.

The compositions of the substitution maps are indexed by rooted trees describing the combinatorial type of the (generic) resulting curves. In particular, each stratum of $\ol{M_{0,n+1}}$ is just the image of the open stratum of an appropriate product $\prod\ol{M_{0,m+1}}$ under some composition of substitution maps.

\subsection{Charts on $\wt{M_{0,n+1}}$.} We will use the following set of charts which form an atlas on $\widetilde{M_{0,n+1}}$. Let $T$ be a tree as above and $\sigma$ be an ordering of its leaves. We call $\sigma$ \emph{compatible} with $T$ if there is an embedding of $T$ into the real plane such that all inner vertices of the tree are in the lower halfplane, all leaves are on the horizontal line $y=0$ and the $x$-coordinates of them are ordered according to $\sigma$.

To any binary rooted tree $T$ compatible with the ordering $\sigma$ one can assign a set of coordinates in appropriate neighborhood $U_{T,\sigma}$ of the corresponding $0$-dimensional stratum $\ul{z}_T$. Let $<$ be the partial ordering of the vertices of $T$ with the root being the minimal element. Let $I(i,j)$ be the maximal inner vertex comparable with the both leaves $z_i$ and $z_j$. The coordinate ring of the open subset $U_{T,\sigma}\subset\widetilde{M_{0,n+1}}$ is generated by the functions $\frac{z_i-z_j}{z_k-z_l}$ for all $i,j$ such that $I(i,j)\not<I(k,l)$ and by $z_i-z_j$ for all $i,j$. We choose the coordinates $u_I$ on $U_{T,\sigma}$ indexed by inner vertices $I$ of the tree $T$ recursively as follows. Let $l(I)\in\{1,\ldots,n\}$ be such that $\sigma(l(I))$ is the maximal index of the $z_i$'s in the left branch at the vertex $I$. Analogously, define $r(I)\in\{1,\ldots,n\}$ such that $\sigma(r(I))$ the minimal index of the $z_i$'s in the right branch at the vertex $I$. For the root vertex $I_0$, we set $u_{I_0}:=z_{r(I_0)}-z_{l(I_0)}$; for any other vertex $I$ let $I_0,I_1,\ldots,I_k=I$ be the shortest way from the root to $I$, then $u_I:=(z_{r(I)}-z_{l(I)})\prod\limits_{j=0}^{k-1}u_{I_j}^{-1}$. Equivalently, $u_I:=\frac{z_{r(I)}-z_{l(I)}}{z_{r(I')}-z_{l(I')}}$ where $I'$ is the preceding vertex (i.e. $I':=\max\{J\in T\ |\ J<I\}$).

Let us describe the stratum $\wt{M_{T'}}\subset\wt{M_{0,n+1}}$ corresponding to a rooted tree $T'$ in the local coordinates determined by a binary rooted tree $T$. The following is clear from the definitions:

\begin{prop}\label{prop-strata-in-coordinates} The stratum $\widetilde{M_{T'}}$ has a nonempty intersection with $U_{T,\sigma}$ if and only if $T'$ is obtained from $T$ by contracting some edges. In the latter case, $\widetilde{M_{T'}}$ is a subset of $U_{T,\sigma}$ defined as follows: $u_I\ne0$ if the (unique) edge of $T$ which ends at $I$ is contracted in $T'$, and $u_I=0$ else.
\end{prop}

\begin{rem} The space $\wt{M_{0,n+1}}$ can be regarded as a closure of the complement of the hyperplane arrangement in $\BC^{n-1}$ formed by the hyperplanes $\{z_i=z_j\}$ for all $i,j$. De Concini and Procesi generalized this construction to any hyperplane arrangement. Namely, in \cite{DCP} they construct the \emph{wonderful closure} of the complement to any hyperplane arrangement, which is smooth and whose boundary is a divisor with normal crossings. They also defined the set of charts generalizing $U_{T,\sigma}$.
\end{rem}

\subsection{Real locus of $\ol{M_{0,n+1}}$.} The space $\overline{M_{0,n+1}}$ is a projective algebraic variety defined over any field (in fact it is defined over $\BZ$), hence we can consider the real loci $\ol{M_{0,n+1}}(\BR)$ and $\widetilde{M_{0,n+1}}(\BR)$ of the spaces $\ol{M_{0,n+1}}$ and $\widetilde{M_{0,n+1}}$, respectively. Note that the space ${\rm Conf}_n(\BR)$ is disconnected, and the connected components are the chambers $D_\sigma:=\{(z_1,\ldots,z_n)\ |\ z_{\sigma(1)}<\ldots<z_{\sigma(n)}\}$ for all permutations $\sigma\in S_n$. We have the atlas on $\widetilde{M_{0,n+1}}(\BR)$ formed by the same charts $U_{T,\sigma}$.

\begin{rem} The open set $U_{T,\sigma}^+:=\{(u_I)\in U_{T,\sigma}\ |\ u_I>0\ \forall I\}$ is the chamber $D_\sigma:=z_{\sigma(1)}<\ldots<z_{\sigma(n)}$ in ${\rm Conf}_n$.
\end{rem}

The space $\wt{M_{0,n+1}}$ can be described as a cell complex. The cells of the codimension $k$ are indexed by pairs $(T,\sigma)$ where $T$ is a rooted tree (not necessarily binary) with $k$ inner vertices and $n$ leaves colored by $z_1,\ldots,z_n$, and $\sigma$ is a compatible ordering of its leaves up to the following equivalence. Two orderings are equivalent if one is obtained from another by reversing the order of the descendants of any inner vertex of $T$, except the root. The closure poset structure on the pairs $(T,\sigma)$ is defined as follows: $(T,\sigma)\le (T',\sigma')$ if $T'$ is obtained from $T$ by contracting some edges and $\sigma$ is equivalent to $\sigma'$ with respect to $T$. In particular, the maximal elements of this poset are indexed by the symmetric group $S_n$, and the corresponding open cells are $D_\sigma$. Two open cells $D_\sigma$ and $D_{\sigma'}$ have a common codimension one face if and only if $\sigma'\sigma^{-1}$ is an involution in $S_n$ which reverses some segment $\{p,p+1,\ldots,q-1,q\}\subset\{1,\ldots,n\}$. Clearly, for any neighboring $D_\sigma$ and $D_{\sigma'}$, there is a tree $T$ compatible with both $\sigma$ and $\sigma'$ such that $\sigma'\sigma^{-1}$ reverses the order of all descendants of some inner vertex $I\in T$. In particular, the differential at $0\in U_{T,\sigma}$ of the gluing function $\varphi:U_{T,\sigma}\to U_{T,\sigma'}$ is just changing the sign of $u_I$.

The vertices of $\wt{M_{0,n+1}}$ correspond to \emph{binary} rooted trees with a compatible ordering of leaves up to equivalence. The edges of $\wt{M_{0,n+1}}$ then correspond to almost binary trees (with exactly one $4$-valent inner vertex).

\begin{rem} \emph{In \cite{HK,Kap} the same cell complex is described in (equivalent) terms of \emph{ordered bracketings}. The cells of the codimension $k$ are indexed by \emph{ordered bracketings} of the product $x_1x_2\ldots x_n$, i.e. pairs consisting of a permutation $\sigma\in S_n$ and a partial bracketing of the product $x_{\sigma(1)}x_{\sigma(2)}\ldots x_{\sigma(n)}$ with $k$ pairs of brackets, up to the equivalence relation. Two bracketings are equivalent if one is obtained from another by reversing the ordering inside any pair of brackets, for example $(x_1x_2(x_3x_4))(x_5x_6)$ is equivalent to $((x_3x_4)x_2x_1)(x_6x_5)$. The closure poset structure on the equivalence classes of bracketings is defined as follows: for equivalence classes of ordered bracketings $\alpha,\beta$ one has $\alpha\le \beta$ if there are representatives $a,b$ of $\alpha,\beta$, respectively, such that $a$ is obtained from $b$ by inserting some pairs of brackets. The vertices of $\wt{M_{0,n+1}}$ are indexed by equivalence classes of \emph{complete} ordered bracketings.}
\end{rem}

\subsection{Cactus group.} One defines the fundamental groupoid of $\wt{M_{0,n+1}}(\BR)$ as follows. The objects are the components of the open stratum of $\wt{M_{0,n+1}}(\BR)$ which are the chambers $D_\sigma$ for all $\sigma\in S_n$. The mophisms from $D_\sigma$ to $D_{\sigma'}$ are the homotopy classes of paths which connect some inner points of the components $D_\sigma$ and $D_{\sigma'}$ and cross the strata of codimension $1$ transversely. Since the symmetric group $S_n$ acts simply transitively on the chambers, this groupoid is in fact the orbifold fundamental group of $\widetilde{M_{0,n+1}}(\BR)/S_n$. Denote this group by $J_n$. Clearly, the group $J_n$ is generated by the homotopy classes of paths connecting \emph{neighboring} open cells (i.e. the open cells having common face of codimension $1$). Thus there are the following generators of $J_n$.

For positive integers $p\le q$, denote by $[p,q]$ the set $\{p,p+1,\ldots,q-1,q\}$. Let $\ol{s_{p,q}}\in S_n$ be the involution reversing the segment $[p,q]\subset[1,n]$. The chambers $D_\sigma$ and $D_{\sigma'}$ are neighboring if $\sigma'\sigma^{-1}$ is $\ol{s_{p,q}}$ for some $p\le q$. Denote by $s_{p,q}$ the element of $J_n$ corresponding to the shortest path from $D_\sigma$ to $D_{\sigma'}$. Then the elements $s_{p,q}$ with $1\le p<q\le n$ generate $J_n$ and the defining relations are
\begin{equation}
\begin{array}{l}s_{p,q}^2=e;\\
s_{p_1,q_1}s_{p_2,q_2}=s_{p_2,q_2}s_{p_1,q_1}\ \text{if}\ q_1<p_2;\\
s_{p_1,q_1}s_{p_2,q_2}s_{p_1,q_1}=s_{p_1+q_1-q_2,p_1+q_1-p_2}\ \text{if}\ p_1\le p_2<q_2\le q_1.
\end{array}
\end{equation}
We refer the reader to \cite{HK} and \cite{DJS} for more details.

The fundamental group $PJ_n:=\pi_1(M_{0,n+1}(\BR))=\pi_1(\widetilde{M_{0,n+1}}(\BR))$ is the kernel of the natural homomorphism $J_n\to S_n$ which maps $s_{p,q}$ to $\ol{s_{p,q}}$. By analogy with braid groups, $PJ_n$ is called the \emph{pure cactus group}. In fact $\widetilde{M_{0,n+1}}(\BR)$ is a $K(\pi,1)$ space for this group, see \cite{DJS}.

We will also use another set of generators of $J_n$, namely, for $k\le l<m$ let
\begin{equation}
s_{[k,l,m]}:=s_{k,m}s_{k,l}s_{l+1,m}.
\end{equation}
Under the natural homomorphism $J_n\to S_n$, the generators $s_{[k,l,m]}$ go to the permutation transposing the segments $[k,l]$ and $[l+1,m]$.

\section{Gaudin subalgebras}\label{sect-bethe-alg}

\subsection{Notation.} For a semisimple $\fg$, we denote by $\fh, X, X^\vee, \Delta, \Delta_+, \Pi_+$ its Cartan subalgebra, weight lattice, coweight lattice, root system, set of positive roots and set of simple roots, respectively. We fix an invariant inner product $(\cdot,\cdot)$ on $X$ such that $(\alpha,\alpha)=2$ for short roots $\alpha\in\Delta$. This determines an invariant inner product on $\fg$ which we also denote by $(\cdot,\cdot)$. We set the Casimir element $C=\sum\limits_{a=1}^{\dim\fg} x_a^2\in U(\fg)$ where $\{x_a\}$ is an orthogonal basis of $\fg$. We denote by $c(\l)$ the eigenvalue of the Casimir operator of $\fg$ on $V_\l$, the irreducible representation with the highest weight $\l$. In particular, for $\fg=\fsl_2$ we have $c(\l)=\frac{\l(\l+2)}{2}$.

For any subset $M\subset\{1,2,\ldots,n\}$ denote by $\Delta_M$ the diagonal embedding of $U(\fg)$ into the tensor product of the $i$-th copies of $U(\fg)$ for all $i\in M$, so for $x\in\fg$ we have $\Delta_M(x)=\sum\limits_{i\in M}x^{(i)}$. For $M\subset \{1,2,\ldots,n\}$ we denote by $C_M$ the image of the Casimir element $C$ under the homomorphism $\Delta_M: U(\fg)\to U(\fg)^{\otimes n}$.

\subsection{} Let $\ul{z}=(z_1,\ldots,z_n)$ be a collection of pairwise distinct complex numbers. The quadratic Gaudin Hamiltonians are the following commuting elements of the algebra $U(\fg)^{\otimes n}$:
\begin{equation*}
H_i=\sum\limits_{j\neq i}\sum\limits_{a=1}^{\dim\fg}
\frac{x_a^{(i)}x_a^{(j)}}{z_i-z_j}=\sum\limits_{j\neq i}
\frac{C_{ij}-C_i-C_j}{2(z_i-z_j)}.
\end{equation*}
Clearly, $H_i$ commute with the diagonal $\fg$ in $U(\fg)^{\otimes n}$. Let us describe the maximal commutative subalgebra $\A(\ul{z})\subset [U(\fg)^{\otimes n}]^\fg$ containing $H_i$.

\subsection{Example.} Let $\fg=sl_2$ and $e,f,h$ be its standard basis. Then $C=ef+fe+\frac{1}{2}h^2$. The algebra
$\A(\ul{z})$ is generated by $H_i=\sum\limits_{k\neq i}
\frac{e^{(i)}f^{(k)}+f^{(i)}e^{(k)}+\frac{1}{2}h^{(i)}h^{(k)}}{z_i-z_k}$ and $C_i=e^{(i)}f^{(i)}+f^{(i)}e^{(i)}+\frac{1}{2}h^{(i)}h^{(i)}$ for $i=1,\ldots,n$ (the latter are the generators of the center of $U(\fg)^{\otimes n}$). The only algebraic relation on the generators $H_i, C_i$ is $\sum\limits_{i=1}^nH_i=0$. Hence the Gaudin algebra is a polynomial algebra with $2n-1$ generators.

\subsection{Construction of the subalgebra $\A(\ul{z})$.} We fix an invariant scalar product on $\fg$ and identify $\fg^*$ with $\fg$ via this scalar product. Consider the infinite-dimensional ind-nilpotent Lie
algebra $\fg_-:=\fg\otimes t^{-1}\BC[t^{-1}]$ -- it is a "half" of
the corresponding affine Kac--Moody algebra $\hat\fg$. The
universal enveloping algebra $U(\fg_-)$ has a natural (PBW) filtration
by the degree with respect to the generators. The associated
graded algebra is the symmetric algebra $S(\fg_-)$ by the
Poincar\'e--Birkhoff--Witt theorem.

There is a natural grading on the associative algebras $S(\fg_-)$ and $U(\fg_-)$ determined by the derivation $L_0$ defined by
\begin{equation}\label{der2}
L_0(g\otimes t^{m})=mg\otimes t^{m}\quad\forall g\in\fg,
m=-1,-2,\dots
\end{equation}

There is also a derivation $L_{-1}$ of degree $-1$ with respect to this grading:
\begin{equation}\label{der3}
L_{-1}(g\otimes t^{m})=mg\otimes t^{m-1}\quad\forall g\in\fg,
m=-1,-2,\dots
\end{equation}

Let $i_{-1}:S(\fg)\hookrightarrow S(\fg_-)$ be the embedding,
which maps $g\in\fg$ to $g\otimes t^{-1}$. The algebra of invariants,  $S(\fg)^{\fg}$, is known to be a free commutative algebra with $\rk\fg$ generators. Let $\Phi_l,\ l=1,\dots,\rk\fg$
be some set of free generators of the algebra $S(\fg)^{\fg}$. The following result is due to Boris Feigin and Edward Frenkel, see \cite{Fr2} and references therein.

\begin{thm}\label{thm-feigin-frenkel}
There exist commuting elements $S_l\in U(\fg_-)$, homogeneous with respect to $L_0$, such that $\gr S_l=i_{-1}(\Phi_l)$. Moreover, the elements $L_{-1}^kS_l$ pairwise commute for all $k\in\BZ_+$ and $l=1,\dots,\rk\fg$.
\end{thm}

Let $U(\fg)^{\otimes n}$ be the tensor product of $n$ copies of
$U(\fg)$. We denote the subspace $1\otimes\dots\otimes
1\otimes\fg\otimes 1\otimes\dots\otimes 1\subset U(\fg)^{\otimes
n}$, where $\fg$ stands at the $i$th place, by $\fg^{(i)}$.
Respectively, for any $x\in U(\fg)$ we set
\begin{equation}
x^{(i)}=1\otimes\dots\otimes 1\otimes x\otimes
1\otimes\dots\otimes 1\in U(\fg)^{\otimes n}.
\end{equation}

Let $\Delta_{[1,n]}:U(\fg_-)\hookrightarrow U(\fg_-)^{\otimes n}$ be the
diagonal embedding (i.e. for $x\in\fg_-$, we have $\Delta_{[1,n]}(x)=\sum\limits_{i=1}^nx^{(i)}$). To any nonzero $w\in\BC$, we assign the homomorphism $\phi_w: U(\fg_-)\to U(\fg)$ of evaluation at the point $w$ (i.e., for $g\in\fg$, we have $\phi_w(g\otimes
t^m)=w^mg$). For any collection of pairwise distinct nonzero
complex numbers $z_i, i=1,\dots,n$, we have the following
homomorphism:
\begin{equation}
\phi_{w_1,\dots,w_n}=(\phi_{w_1}\otimes\dots\otimes\phi_{w_n})\circ
\Delta_{[1,n]}:U(\fg_-)\to U(\fg)^{\otimes n}.
\end{equation}
More explicitly, we have
$$\phi_{w_1,\dots,w_n}(g\otimes
t^m)=\sum\limits_{i=1}^nw_i^mg^{(i)}.$$

Consider the following $U(\fg)^{\otimes n}$-valued functions in
the variable $w$
$$
S_l(w;z_1,\dots,z_n):=\phi_{w-z_1,\dots,w-z_n}(S_l).
$$

We define the  Gaudin subalgebra $\A(\ul{z})\subset U(\fg)^{\otimes n}$ as a subalgebra generated by $S_l(w;z_1,\dots,z_n)$ for all $w\in\BC\backslash\{z_1,\ldots,z_n\}$. Due to Theorem~\ref{thm-feigin-frenkel}, this subalgebra is commutative. The subalgebra $\A(\ul{z})\subset U(\fg)^{\otimes n}$ is also known as \emph{Bethe algebra}.

Let $S_l^{i,m}(z_1,\dots,z_n)$ be the coefficients of the principal
part of the Laurent series of $S_l(w;z_1,\dots,z_n)$ at the point
$z_i$, i.e.,
$$S_l(w;z_1,\dots,z_n)=\sum\limits_{m=1}^{m=\deg\Phi_l}S_l^{i,m}(z_1,\dots,z_n)
(w-z_i)^{-m}+O(1)\
\text{as}\ w\to z_i.$$

Taking the generator $S_l$ corresponding to the quadratic Casimir element on $S(\fg)$, one gets the quadratic Gaudin Hamiltonians (\ref{quadratic}) as the residues of $S_l(w;z_1,\dots,z_n)$ at the points $z_1,\ldots,z_n$. The following result is well-known (see e.g. \cite{CFR} for the proof).

\begin{prop}\label{generators2}\cite{CFR}
\begin{enumerate}
 \item The elements
$S_l^{i,m}(z_1,\dots,z_n)\in U(\fg)^{\otimes n}$ are homogeneous under
simultaneous affine transformations of the parameters $z_i\mapsto
az_i+b$ (i.e. $S_l^{i,m}(az_1+b,\dots,az_n+b)$ is proportional to $S_l^{i,m}(z_1,\dots,z_n)$). \item The subalgebra $\A(\ul{z})$ is a free commutative
algebra generated by the elements $S_l^{i,m}(z_1,\dots,z_n)\in
U(\fg)^{\otimes n}$, where $i=1,\dots,n-1$, $l=1,\dots,\rk\fg$,
$m=1,\dots,\deg\Phi_l$, and by
$S_l^{n,\deg\Phi_l}(z_1,\dots,z_n)\in U(\fg)^{\otimes
n}$, where $l=1,\dots,\rk\fg$. \item All the elements of $\A(\ul{z})$ are
invariant with respect to the diagonal action of $\fg$. \item The
center of the diagonal $\Delta_{[1,n]}(U(\fg))\subset U(\fg)^{\otimes n}$
is contained in $\A(\ul{z})$.
\end{enumerate}
\end{prop}

\begin{rem}\label{rem-diagcenter} It is easy to see that one can replace $S_l^{n,\deg\Phi_l}(z_1,\dots,z_n)$ in (2) by the generators of the center of $\Delta_{[1,n]}(U(\fg))$.
\end{rem}

\subsection{Operad structure on commutative subalgebras.}
For any partition of the set $\{1,2,\ldots,n\}=M_1\cup\ldots\cup M_k$, define the homomorphism $$D_{M_1,\ldots,M_k}:U(\fg)^{\otimes k}\hookrightarrow
U(\fg)^{\otimes n},$$ taking $x^{(i)}\in U(\fg)^{\otimes k}$, for $x\in\fg,\ i=1,\ldots,k$, to $\sum\limits_{j\in M_i}x^{(j)}$.

For any subset $M=\{j_1,\ldots,j_m\}\subset\{1,2,\ldots,n\}$, with $j_1<\ldots<j_m$, let $I_M:U(\fg)^{\otimes m}\hookrightarrow
U(\fg)^{\otimes n}$ be the embedding of the tensor product of the copies of $U(\fg)$ indexed by $M$, i.e. $I_M(x^{(i)}):=x^{(j_i)}\in U(\fg)^{\otimes n}$ for any $x\in\fg,\ i=1,\ldots,m$. Clearly, all these homomorphisms are $\fg$-equivariant and every element in the image of $D_{M_1,\ldots,M_k}$ commutes with every element of $I_{M_i}([U(\fg)^{\otimes m_i}]^\fg)$ for $i=1,\ldots,k$. This gives us the following ``substitution'' homomorphism defining an operad structure on the spaces $[U(\fg)^{\otimes n}]^\fg$
\begin{equation}\label{operad-Ug}
\gamma_{k;M_1,\ldots,M_k}=D_{M_1,\ldots,M_k}\otimes\bigotimes\limits_{i=1}^k I_{M_i}:[U(\fg)^{\otimes k}]^\fg\otimes\bigotimes\limits_{i=1}^k[U(\fg)^{\otimes m_i}]^\fg\to[U(\fg)^{\otimes n}]^\fg.
\end{equation}

Let ${\rm Subalg}_{n}$ be the set of commutative subalgebras in $U(\fg)^{\otimes n}$ of the transcendence degree $\frac{n-1}{2}\dim\fg+\frac{n+1}{2}\rk\fg$ commuting with the diagonal $\fg$ and containing the center of $U(\fg)^{\otimes n}$ and of the diagonal $U(\fg)$.

\begin{prop}\label{prop-substitution} The homomorphism~(\ref{operad-Ug}) defines a substitution map $$\gamma_{k;M_1,\ldots,M_k}:{\rm Subalg}_k\times\prod\limits_{i=1}^k{\rm Subalg}_{m_i}\to {\rm Subalg}_{n}.$$ Moreover, $\gamma_{k;M_1,\ldots,M_k}(\A(\ul{w});\A(\ul{u_k}),\ldots,\A(\ul{u_k}))$ has the same Poincar\'e series as $\A(\ul{z})$.
\end{prop}

\begin{proof} It is easy to check that the resulting subalgebra commutes with the diagonal $\fg$ and contains the center of $U(\fg)^{\otimes n}$ and of the diagonal $U(\fg)$. To check that it has right transcendence degree and Poincar\'e series, we need the following

\begin{lem}\label{lem-factorization} \begin{enumerate} \item The homomorphism $\gamma_{k;M_1,\ldots,M_k}$ factors as
$$
\gamma_{k;M_1,\ldots,M_k}:[U(\fg)^{\otimes k}]^\fg\otimes\bigotimes\limits_{i=1}^k[U(\fg)^{\otimes m_i}]^\fg\to [U(\fg)^{\otimes k}]^\fg\otimes_{ZU(\fg)^{\otimes k}}\bigotimes\limits_{i=1}^k[U(\fg)^{\otimes m_i}]^\fg\hookrightarrow[U(\fg)^{\otimes n}]^\fg.
$$
\item The algebras $[U(\fg)^{\otimes k}]^\fg$, $\bigotimes\limits_{i=1}^k[U(\fg)^{\otimes m_i}]^\fg$ and $[U(\fg)^{\otimes n}]^\fg$ are free as $ZU(\fg)^{\otimes k}$-modules.
\end{enumerate}
\end{lem}

\begin{proof} By Kostant's theorem \cite{Kos}, $U(\fg)$ is a free $ZU(\fg)$-module, and one can choose a $\fg$-invariant space of generators for it. Hence $[U(\fg)^{\otimes k}]^\fg$ and $[U(\fg)^{\otimes n}]^\fg$ are both free as $ZU(\fg)^{\otimes k}$-modules.

For the rest, it suffices to show that $[U(\fg)^{\otimes m}]^\fg$ is a free $\Delta_{[1,m]}(ZU(\fg))$-module and that the homomorphism $\Delta_{[1,m]}\cdot\id:U(\fg)\otimes [U(\fg)^{\otimes m}]^\fg\to U(\fg)^{\otimes m}$ factors as $U(\fg)\otimes [U(\fg)^{\otimes m}]^\fg\to U(\fg)\otimes_{ZU(\fg)} [U(\fg)^{\otimes m}]^\fg \hookrightarrow U(\fg)^{\otimes m}$. But this is a particular case of Knop's theorem on Harish-Chandra map for reductive group actions (see \cite{Kn}, items (d) and (e) of the Main Theorem). Indeed, the Main Theorem of \cite{Kn} states that for any reductive $G$ and any smooth affine $G$-variety $X$ the algebra $D(X)^G$ of $G$-invariant differential operators and its commutant $U(X)$ in the algebra $D(X)$ are both free modules over the center of $D(X)^G$. Moreover, the product $D(X)^G\cdot U(X)\subset D(X)$ is the tensor product $D(X)^G\otimes_{ZD(X)^G} U(X)$ of $D(X)^G$ and $U(X)$ over the center of $D(X)^G$. To get the desired statement we just apply this to the $G^{\times (m+1)}$-action on $X=G^{\times m}$, where $e\times G^{\times m}$ acts on $G^{\times m}$ from the right and $G\times e^{\times m}$ acts on $G^{\times m}$ diagonally from the left. Indeed, for this case we have $D(X)^{G^{\times (m+1)}}$ is $[U(\fg)^{\otimes m}]^\fg$. By Theorem 10.1 of \cite{Kn}, the center of $[U(\fg)^{\otimes m}]^\fg$ is $ZU(\fg)^{\otimes (m+1)}=\Delta_{[1,m]}(ZU(\fg))\otimes ZU(\fg)^{\otimes m}$ and the algebra $U(X)$ contains $\Delta_{[1,m]}(U(\fg))\otimes ZU(\fg)^{\otimes m}$ as the subalgebra of $e\times G^{\times m}$-invariants. Hence we have \begin{multline*}\Delta_{[1,m]}(U(\fg))\cdot [U(\fg)^{\otimes m}]^\fg=(\Delta_{[1,m]}(U(\fg))\otimes ZU(\fg)^{\otimes m})\otimes_{ZU(\fg)^{\otimes (m+1)}}[U(\fg)^{\otimes m}]^\fg=\\=U(\fg)\otimes_{ZU(\fg)} [U(\fg)^{\otimes m}]^\fg.\end{multline*}

\end{proof}

The subalgebras $D_{M_1,\ldots,M_k}(\A(\ul{w}))$ and $\bigotimes\limits_{i=1}^{k}I_{M_i}(\A(\ul{u_1})\otimes\ldots\otimes\A(\ul{u_k}))$ both contain $D_{M_1,\ldots,M_k}(ZU(\fg)^{\otimes k})$. Moreover both of these subalgebras are free modules over $D_{M_1,\ldots,M_k}(ZU(\fg)^{\otimes k})$. By Lemma~\ref{lem-factorization}, $\gamma_{k;M_1,\ldots,M_k}(\A(\ul{w});\A(\ul{u_1}),\ldots,\A(\ul{u_k}))$ is the tensor product of $D_{M_1,\ldots,M_k}(\A(\ul{w}))$ and $\bigotimes\limits_{i=1}^{k}I_{M_i}(\A(\ul{u_1})\otimes\ldots\otimes\A(\ul{u_k}))$ over $D_{M_1,\ldots,M_k}(ZU(\fg)^{\otimes k})$. Hence both $\gamma_{k;M_1,\ldots,M_k}(\A(\ul{w});\A(\ul{u_1}),\ldots,\A(\ul{u_k}))$ and $\A(\ul{z})$ are polynomial algebras with $(n-1)\deg\Phi_l+1$ generators of degree $\deg\Phi_l$ for each central generator $\Phi_l$. This proves the statement on the Poincar\'e series.
\end{proof}

\subsection{Closure of the family $\A(\ul{z})$.}
From Proposition~\ref{generators2} it follows that the commutative subalgebras $\A(\ul{z})\subset U(\fg)^{\otimes n}$ form a flat family parameterized by the configuration space $M_{0,n+1}$. This means that for any positive integer $N$ the intersection of $\A(\ul{z})$ with the $N$-th filtered component $\rm{PBW}^{(N)}U(\fg)^{\otimes n}$ with respect to the PBW filtration of the universal enveloping algebra $U(\fg)^{\otimes n}$ has the same dimension $d(N)$. Hence there is a regular map from $M_{0,n+1}$ to the product of the Grassmannians $\prod\limits_{M\le N}{\rm Gr}(d(M),{\rm PBW}^{(M)}U(\fg)^{\otimes n})$ taking $\ul{z}\in M_{0,n+1}$ to $\prod\limits_{M\le N}\A(\ul{z})\cap {\rm PBW}^{(M)}U(\fg)^{\otimes n})$. Let $Z_N$ be the closure of the image of this map. Then there are surjective restriction maps $r_{NM}:Z_N\to Z_M$ for any $M<N$. The inverse limit $Z=\lim\limits_{\leftarrow} Z_N$ is well-defined as a pro-algebraic scheme. The restriction of the tautological vector bundle on the Grassmannian gives a sheaf $\A$ of commutative algebras on $Z$

\begin{prop} The fiber of $\A$ at any point of $Z$ is a commutative subalgebra in $U(\fg)^{\otimes n}$ which has the same Poincar\'e series as $\A(\ul{z})$ and coincides with $\A(\ul{z})$ for all $\ul{z}\in M_{0,n+1}\subset Z$.
\end{prop}

\begin{proof} By definition, $\A$ is a sheaf of filtered vector spaces with the same Poincar\'e series as $\A(\ul{z})$ which coincides with $\A(\ul{z})$ for all $\ul{z}\in M_{0,n+1}\subset Z$. The conditions of being closed under the associative product on $U(\fg)^{\otimes n}$ and of being commutative are Zariski-closed on the Grassmannian ${\rm Gr}(d(N),{\rm PBW}^{(N)}U(\fg)^{\otimes n})$), hence each fiber of $\A$ is a commutative subalgebra in $U(\fg)^{\otimes n}$.
\end{proof}

So we have a flat family of commutative subalgebras $\A(\ul{z})\subset U(\fg)^{\otimes n}$ parametrized by $\ul{z}\in Z$, and in fact $Z$ is the indexing space for all possible limiting subalgebras of the family $\A(\ul{z})$. The construction of $Z$ is very general (it is well defined for any flat family of subspaces in a filtered space), but generally such scheme does not seem to have any good properties. Contrary, in our case it turns to be a smooth algebraic scheme:

\begin{thm}\label{thm-closure} \begin{enumerate} \item $Z$ is a smooth algebraic variety isomorphic to $\overline{M_{0,n+1}}$;
\item for any $\ul{z}\in\overline{M_{0,n+1}}$ the corresponding commutative subalgebra $\A(\ul{z})\subset U(\fg)^{\otimes n}$ is a polynomial algebra with $\frac{n-1}{2}\dim\fg+\frac{n+1}{2}\rk\fg$ generators;
\item the operad structures on $\overline{M_{0,n+1}}$ and on ${\rm Subalg}_n$ match, i.e. we have $$\A(\gamma_{k;M_1,\ldots,M_k}(\ul{w};\ul{u_1},\ldots,\ul{u_k}))=
    \gamma_{k;M_1,\ldots,M_k}(\A(\ul{w});\A(\ul{u_1}),\ldots,\A(\ul{u_k})).$$
\end{enumerate}
\end{thm}

\begin{rem} \emph{Aguirre, Felder and Veselov proved this in \cite{AFV} for quadratic components of $\A(\ul{z})$ generated by the elements $H_i(\ul{z})$. Also, in \cite{CFR} a set-theoretical version of this Theorem was proved, i.e. all the subalgebras $\A(\ul{z})$ corresponding to boundary points $\ul{z}\in\overline{M_{0,n+1}}$ were explicitly described. Our proof uses the ideas of \cite{CFR}.}
\end{rem}

\begin{proof}
First, we will need the following description of some limits of some generators of the subalgebras $\A(\ul{z})$.

\begin{lem}\label{lem-predel} \begin{enumerate} \item For $l=1,\dots,\rk\fg$, $m=1,\ldots,\deg S_l$, $k=1,\ldots, n-1$, the elements $\sum\limits_{j=1}^k Res_{w=w_j} w^m S_l(w;w_1,\ldots,w_n)$ are well-defined outside the hyperplanes $\{w_i=w_j\}$ for $1\le i \le k<j\le n$.
\item Suppose that for some $p\in\{1,\ldots, k\}$ we have $w_{1}=\ldots=w_{p}$. Then $\sum\limits_{j=1}^k Res_{w=w_j} w^m S_l(w;w_1,\ldots,w_n)=D_{[1,p],p+1,\ldots,n}(\sum\limits_{j=1}^k Res_{w=w_j} w^m S_l(w;w_p,w_{p+1},\ldots,w_n))$.
\item For any $p\ge k$ the limit $\lim\limits_{w_i\to\infty\ \forall i> p}\sum\limits_{j=1}^k Res_{w=w_j} w^m S_l(w;w_1,\ldots,w_n)$ is well-defined and equals $I_{[1,p]}(S_l(w;w_1,\ldots,w_p))$.
\end{enumerate}
\end{lem}

\begin{proof} The first two assertions are obvious since the homomorphism $\phi_{w-w_1,\dots,w-w_n}$ is well-defined for all $w_i$ and when some of the $w_i$'s coincide it just factors through the corresponding diagonal embedding. The third one follows from the following obvious
\begin{lem}\label{lem-infty-limit}\cite{Ryb1}
The limit $\lim\limits_{z\to\infty}\phi_z$ is the counit map $\ep:U(\fg_-)\to\BC\cdot1\subset~U(\fg)$.
\end{lem} Indeed, we have
\begin{multline*}
\lim\limits_{w_i\to\infty}S_l(w;w_1,\ldots,w_n)=
\lim\limits_{w_i\to\infty}\phi_{w-w_1,\ldots,w-w_n}(S_l)=\\
=(\phi_{w-w_1}\otimes\dots\otimes\phi_{w-w_k}\otimes\ep\otimes\ldots\otimes\ep)\circ
\Delta_{\{1,\ldots,n\}}(S_l)=I(S_l(w;w_1,\ldots,w_n)).
\end{multline*}.

\end{proof}

To prove the first two assertions of the Theorem, we produce, for any planar binary rooted tree $T$, a set of generators of $\A(\ul{z})$ that are regular on $U_{T,\sigma}$ and algebraically independent at any point of $U_{T,\sigma}$. We can assume without loss of generality that $\sigma=e$.

Let us introduce some notation. We denote by $\le$ the partial order on the set of vertices of $T$, where the root vertex is minimal. To any inner vertex $I\in T$ we assign the subsets $L(I), R(I), LR(I)\subset [1,n]$ formed by all leaves of $T$ on the left branch of $T$ at $I$, right branch of $T$ at $I$ and on both of them, respectively. We denote by $l(I)$ the maximum of $L(I)$ and by $r(I)$ the minimum of $R(I)$.

Let $I$ be an inner vertex of the tree $T$. Set $k=l(I)$ then $r(I)=k+1$. To the vertex $I$ and a number $m=1,\ldots,\deg S_l$ we assign the element $S_{l,T,I}^{(m)}\in\A(\ul{z})$ defined as
\begin{multline*}
S_{l,T,I}^{(m)}:=\sum\limits_{j\in L(I)}Res_{w=\frac{z_j-z_k}{z_{k+1}-z_k}} w^m S_l(w;\frac{z_1-z_k}{z_{k+1}-z_k},\ldots,\frac{z_k-z_k}{z_{k+1}-z_k},\frac{z_{k+1}-z_k}{z_{k+1}-z_k},\dots,\frac{z_n-z_k}{z_{k+1}-z_k})=\\=\sum\limits_{j\in L(I)}Res_{w=\frac{z_j-z_k}{\prod\limits_{J\le I}u_J}} w^m S_l(w;\frac{z_1-z_k}{\prod\limits_{J\le I}u_J},\ldots,\frac{z_k-z_k}{\prod\limits_{J\le I}u_J},\frac{z_{k+1}-z_k}{\prod\limits_{J\le I}u_J},\dots,\frac{z_n-z_k}{\prod\limits_{J\le I}u_J}).
\end{multline*}
where $L(I)$ stands for the left branch of $T$ at $I$.
By Proposition~\ref{generators2}, the elements $S_{l,T,I}^{(m)}\in\A(\ul{z})$ together with the generators of the center of $\Delta_{[1,n]}U(\fg)\subset U(\fg)^{\otimes n}$ are algebraically independent and generate $\A(\ul{z})$ for $\ul{z}\in U_{T,e}\backslash\{\exists I : u_I=0\}$ (i.e. on the intersection of $U_{T,e}$ with the open stratum of $\wt{M_{0,n+1}}$).

Now we can compute $S_{l,T,I}^{(m)}$ at $\ul{z}\in U_{T,\sigma}$ such that some of the $u_J$'s vanish, i.e. $\ul{z}\in \wt{M_{T'}}\subset U_{T,\sigma}$ for some non-binary tree $T'$. Let $T(I)$ be the subtree containing $I$ and bounded by all inner vertices $J$ such that $u_J=0$. Let $J_0$ be the root of $T(I)$ and $J_1,\ldots,J_k$ be the leaves of $T(I)$, then by Lemma~\ref{lem-predel} $S_{l,T,I}^{(m)}=I_{LR(J_0)}\circ D_{LR(J_1),\ldots,LR(J_k)}(S_{T(I),I}^{(m)})$ (informally, the leaves not from $LR(J_0)$ do not contribute to $S_{l,T,I}^{(m)}$ and the contributions of the leaves from the same $LR(J_i)$ are equal). Hence $S_{l,T,I}^{(m)}\in\A(\ul{z})$ are well-defined. Note that $S_{l,T,I}^{(m)}=I_{LR(J_0)}\circ D_{LR(J_1),\ldots,LR(J_k)}(S_{T(I),I}^{(m)})$ for all $I$ form (together with the central generators of $\Delta_{[1,n]}U(\fg)$) the complete set of generators of the subalgebra obtained by the composition of the operations $\gamma$ according to the tree $T'$. Hence they an algebraically independent system for every $\ul{z}\in\wt{M_{T'}}\subset U_{T,e}$, and the third assertion of the Theorem is also proved.
\end{proof}

Now recall the result of Mukhin, Tarasov and Varchenko:

\begin{thm}\cite{MTV07} Let $\fg=\fsl_N$. For any collection $\ul{\l}$ of dominant integral weights, the space $\BV_{\ul{\l}}^{sing}$ is a cyclic module over $\A(\ul{z})$.
\end{thm}

We generalize this theorem to any $\ul{z}\in \overline{M_{0,n+1}}$.

\begin{thm}\label{thm-cyclic} Let $\fg=\fsl_N$. \begin{enumerate} \item  For any collection $\ul{\l}$ of dominant integral weights and any $\ul{z}\in\overline{M_{0,n+1}}$, the space $\BV_{\ul{\l}}^{sing}$ is a cyclic module over $\A(\ul{z})$.
\item For any collection $\ul{\l}$ of dominant integral weights, the algebra $\A(\ul{z})$ with \emph{real} $\ul{z}$ has simple spectrum in the space $\BV_{\ul{\l}}^{sing}$.
\end{enumerate}
\end{thm}

\begin{proof} To prove the first assertion, we proceed by induction on $n$. Suppose that $\ul{z}=\gamma_{k;M_1,\ldots,M_k}(\ul{w};\ul{u_1},\ldots,\ul{u_k})$. Then by Theorem~\ref{thm-closure} the corresponding subalgebra $\A(\ul{z})$ is generated by $I_{M_i}(\A(\ul{u_i}))$ and $D_{M_1,\ldots,M_k}(\A(\ul{w}))$. Let $\BV_{\ul{\l}}=\bigoplus W_{(\nu)}\otimes V_{(\nu)}$ be the decomposition of $\BV_{\ul{\l}}$ into the sum of isotypic component with respect to $D_{M_1,\ldots,M_k}(\fg^{\oplus k})$ with $V_{(\nu)}$ being the irreducible representation of $\fg^{\oplus k}$ with the highest weight $(\nu)=(\nu_1,\ldots,\nu_k)$ and $W_{(\nu)}:=\Hom_{\fg^{\oplus k}}(V_{(\nu)},\BV_{\ul{\l}})$ being the multiplicity space. By induction hypothesis, the multiplicity spaces $W_{(\nu)}$ are cyclic $\bigotimes\limits_{i=1}^k I_{M_i}(\A(\ul{u_i}))$-modules. On the other hand, the space of singular vectors of each $V_{(\nu)}$ is a cyclic $D_{M_1,\ldots,M_k}(\A(\ul{w}))$-module. Hence the entire module $\BV_{\ul{\l}}^{sing}$ is cyclic with respect to $\A(\ul{z})$.

For real $\ul{z}$ the generators of $\A(\ul{z})$ act by Hermitian operators in any $\BV_{\ul{\l}}$, and hence are diagonalizable, see e.g. Lemma~2 of \cite{FFR10}. Since there is a cyclic vector for the action of $\A(\ul{z})$ in the space $\BV_{\ul{\l}}^{sing}$ the joint eigenvalues of the generators on different eigenvectors are different.
\end{proof}

\begin{cor}\label{cor-covering} For any collection $\ul{\l}$ of dominant integral weights, the spectra of the algebras $\A(\ul{z})$ with \emph{real} $\ul{z}$ in the space $\BV_{\ul{\l}}^{sing}$ form a unbranched covering of $\overline{M_{0,n+1}(\BR)}$.
\end{cor}

\begin{cor}
The pure cactus group $PJ_n$ acts on the spectrum of $\A(\ul{z})$ in $\BV_{\ul{\l}}^{sing}$ for any $\ul{z}\in\overline{M_{0,n+1}(\BR)}$. Moreover, the group $J_n$ acts on spectra of $\A(\ul{z})$ permuting the coordinates of $\ul{z}$.
\end{cor}

\section{Cactus group and crystals}\label{sect-kamnitzer}

We list here some results on crystal commutors due to Henriques, Kamnitzer and Tingley, see \cite{HK,KT} for more details.

\subsection{Crystal bases.} Let $\BC_q$ be the field of rational functions of the formal variable $q^{\frac{1}{2}}$. Consider the quantum group $U_q(\fg)$ corresponding to the Lie algebra $\fg$. It is a Hopf algebra over $\BC_q$ with the standard Chevalley generators $e_i,f_i,q^{h},\ h\in X^\vee$ satisfying the following defining relations
\begin{equation}
\begin{array}{l}
q^0=1,\ q^{h_1+h_2}=q^{h_1}q^{h_2};\\
q^he_iq^{-h}=q^{(\alpha_i,h)}e_i;\\
q^hf_iq^{-h}=q^{-(\alpha_i,h)}f_i;\\
e_if_j-f_je_i=\delta_{ij}\frac{q^{d_ih_i}-q^{-d_ih_i}}{q^{d_i}-q^{-d_i}};
\end{array}
\end{equation}
and $q$-Serre relations, see e.g. \cite{Lu} for details. The comultiplication is defined on the generators of $U_q(\fg)$ as
\begin{equation}
\begin{array}{l}
\Delta(q^h)=q^h\otimes q^h;\\
\Delta(e_i)=e_i\otimes q^{d_ih_i}+1\otimes e_i;\\
\Delta(f_i)=f_i\otimes 1+ q^{-d_ih_i}\otimes f_i.
\end{array}
\end{equation}

The algebra $U_q(\fg)$ over the formal neighborhood of $q=1$ can be regarded as a deformation of the universal enveloping algebra $U(\fg)$ in the class of Hopf algebras. Moreover, for $q$ being not a root of unity, the category of finite-dimensional $U_q(\fg)$-modules is semisimple, and the irreducibles are indexed by the weight lattice of $\fg$. We denote by $V^q_\l$ the irreducible $U_q(\fg)$-module with the highest weight $\l$. It is a flat deformation of the $U(\fg)$-module $V_\l$. Moreover, $V^q_\l$ decomposes into the direct sum of weight spaces (i.e. joint eigenspaces of $q^h$, $h\in X^\vee$), $V^q_\l=\bigoplus V^q_\l(\mu)$. This is a flat deformation of the weight decomposition $V_\l=\bigoplus V_\l(\mu)$.

Define the divided powers of the generators $e_i, f_i$:
\begin{equation}
e_i^{(n)}:=\frac{e_i^n}{[n!]_q},\ f_i^{(n)}:=\frac{f_i^n}{[n!]_q},
\end{equation}
where $[n!]_q:=\prod\limits_{k=1}^n\frac{q^k-q^{-k}}{q-q^{-1}}$.

Let $V^q$ be a finite-dimensional representation of $U_q(\fg)$. Denote by $V^q(\mu)\subset V^q$ the weight space of the weight $\mu$. The Kashiwara operators $\tilde e_i, \tilde f_i$ on $V^q$ are defined as follows. Consider the $U_q(\fsl_2)$ generated by $e_i, f_i$. The space $V^q$ is decomposed into the direct sum of $U_q(\fsl_2)$-modules, $V^q=\oplus M^q_{l}$, where $M^q_l$ is the irreducible $U_q(\fsl_2)$-module of highest weight $l\in\BZ_+$. Each $v\in M^q_l$ represents as $v=f_i^{(m)}v_l$ where $v_l$ is the highest weight vector of $M^q_l$. Then by definition $\tilde f_i v=f_i^{(m+1)}v_l$ and $\tilde e_i v=f_i^{(m-1)}v_l$.

Let $\BC_q^\infty\subset \BC_q$ be the subring of rational functions which are regular at $\infty$ and $\fm_q^\infty\subset\BC_q^\infty$ be the maximal ideal of $\infty$. A \emph{crystal base} of $V^q$ is a pair $(L,\CB)$ satisfying the following conditions:

\begin{enumerate}
\item $L\subset V^q$ is a free $\BC_q^\infty$-module such that $V^q=\BC_q\otimes_{\BC_q^\infty}L$;
\item $\CB$ is a basis of $L/\fm_q^\infty L$;
\item $L=\bigoplus_{\mu}L(\mu)$ and $\CB=\coprod \CB(\mu)$ where $L(\mu)=L\cap V^q(\mu)$, $\CB(\mu)=\CB\cap L(\mu)/\fm_q^\infty L(\mu)$;
\item the operators $\tilde e_i$ and $\tilde f_i$ preserve the lattice $L$ (and hence operate on $L/\fm_q^\infty L$);
\item $\tilde e_i\CB\subset\CB\sqcup\{0\}$ and $\tilde f_i\CB\subset\CB\sqcup\{0\}$;
\item for $b,b'\in\CB$ we have $b'=\tilde e_i b$ if and only if $b=\tilde f_i b'$.
\end{enumerate}

This endows the set $\CB$ with the (purely combinatorial) structure of a \emph{crystal}. That is the set of maps $\tilde e_i, \tilde f_i: \CB\to \CB\sqcup\{0\}$ satisfying certain axioms, see \cite{HK} for details. It is natural to represent crystals as directed colored graphs, whose vertices are the elements of $\CB$ and edges of the $i$-th color are the maps $\tilde e_i$. The crystal base for $V^q$ always exists, and the corresponding crystal $\CB$ is uniquely determined by $V^q$. Thus we get a a category whose objects are crystals of finite-dimensional representations of $U_q(\fg)$ and morphisms are maps respecting the crystal structure. This category is semisimple in the sense that every object $\CB$ is a direct sum (i.e. set-theoretical union) of irreducibles $\CB_\l$ (i.e. crystals of irreducible $U_q(\fg)$-modules $V^q_\l$, $\l\in X(\fg)$), $\#\Hom(\CB_\l,\CB_\l)=1$ and $\Hom(\CB_\l,\CB_\mu)=\emptyset$ for $\l\ne\mu$.

\subsection{Monoidal structure.} Let $(L_1,\CB_1)$ and $(L_2,\CB_2)$ be crystal bases for $U_q(\fg)$-modules $M_1, M_2$ respectively. Then $(L_1\otimes L_2,\CB_1\otimes\CB_2)$ is a crystal base of the tensor product $M_1\otimes M_2$. This gives a structure of a crystal on $\CB_1\times\CB_2$: set $\varphi(b)=\max\{k\ |\ \tilde f_i^kb\ne0\}$, $\varepsilon(b)=\max\{k\ |\ \tilde e_i^kb\ne0\}$, then the maps $\tilde e_i,\tilde f_i$ on $\CB_1\times\CB_2$ are defined as
\begin{equation}
\tilde e_i (b_1\otimes b_2)=\left\{ \begin{array}{lcr} \tilde e_i b_1\otimes b_2 & \text{if} & \varphi_i(b_1)\ge \varepsilon_i(b_2) \\
 b_1\otimes \tilde e_i b_2 & \text{if} & \varphi_i(b_1) < \varepsilon_i(b_2) \end{array} \right.
\end{equation}
\begin{equation}
\tilde f_i (b_1\otimes b_2)=\left\{ \begin{array}{lcr} \tilde e_i b_1\otimes b_2 & \text{if} & \varphi_i(b_1)> \varepsilon_i(b_2) \\
 b_1\otimes \tilde e_i b_2 & \text{if} & \varphi_i(b_1) \le \varepsilon_i(b_2) \end{array} \right.
\end{equation}
Here $b_1\otimes b_2=(b_1,b_2)\in\CB_1\times\CB_2$ for $b_1\in\CB_1, b_2\in\CB_2$, and $b_1\otimes0:=0=:0\otimes b_2$. Thus the category of crystals is naturally a monoidal category.

\subsection{Example: $\fsl_2$ case.} Let $\fg=\fsl_2$. The irreducible $U_q(\fsl_2)$-modules are $V^q_\l$ with $\l\in\BZ_+$. The vectors $f^{(n)}v_\l$ generate a $\BC_q^\infty$-lattice $L_\l\subset V_\l$, and the projections of these vectors to $L_\l/\fm_q^\infty L_\l$ form a crystal basis. Hence the crystal of the irreducible $U_q(\fsl_2)$-module $V_\l^q$ has the form
$$
\bullet\longrightarrow\bullet\longrightarrow\dots\longrightarrow\bullet
$$
The tensor product of two irreducible crystals has the form
$$
\begin{array}{lllllllll}
\bullet & &\vdots& & \bullet  & & \bullet  &\longrightarrow \ldots \longrightarrow & \bullet \\
\downarrow & &\vdots& & \downarrow  & &   & &  \\
\bullet  & &\dots& & \bullet  & \longrightarrow& \bullet  & \longrightarrow  \ldots \longrightarrow & \bullet \\
\downarrow & &   & &   & &    \\
\vdots & &   & &   & &    \\
\downarrow & &   & &   & &    \\
\bullet   & \longrightarrow& \dots & \longrightarrow& \bullet  & \longrightarrow& \bullet  & \longrightarrow \ldots \longrightarrow & \bullet
\end{array}
$$
Notice that the tensor products of given two $\fsl_2$-crystals in different order are isomorphic, but the isomorphism is not the transposition of the multiples.

\subsection{Braiding.} The category of $U_q(\fg)$-modules is braided, i.e. for any pair $M_1,M_2$ of $U_q(\fg)$-modules there is an isomorphism $R:M_1\otimes M_2\to M_2\otimes M_1$ which is functorial and satisfies the braid (Yang-Baxter) relation for any triple of $U_q(\fg)$-modules. This isomorphism is not an involution: for irreducible finite-dimensional $U_q(\fg)$-modules $V_{\l_1}^q$ and $V_{\l_2}^q$ , the operator $R^2:V_{\l_1}^q\otimes V_{\l_2}^q\to V_{\l_1}^q\otimes V_{\l_2}^q$ acts on each irreducible component isomorphic to $V_\l^q$ as the scalar $q^{c(\l)-c(\l_1)-c(\l_2)}$ where $c(\l)=(\l,\l+2\rho)$ is the value of the Casimir operator on the $\fg$-module $V_\l$. In particular, $R^2$ does not preserve any $\BC_q^{\infty}$-lattice in $V_{\l_1}^q\otimes V_{\l_2}^q$ and hence does not give any braiding on the category of crystals. In fact, there is no structure of a braided category on the category of crystals, see \cite{S}. However, there is a \emph{coboundary} category structure of on the category of $\fg$-crystals.

\subsection{Coboundary categories.} A \emph{coboundary category} is a monoidal category $\CalC$ along with natural isomorphisms $ s_{X,Y}: X\otimes Y \to Y \otimes X$ for all
$X, Y \in Ob\ \CalC$ satisfying the following conditions:
\begin{enumerate}
\item $s_{X,Y}\circ s_{Y,X}=\Id$, and
\item the \emph{cactus relation}: for all triples $X,Y,Z \in Ob\ \CalC$, the diagram
$$
\begin{CD}
X\otimes Y\otimes Z @>s_{X,Y}\otimes1>> Y\otimes X\otimes Z  \\
@VV1\otimes s_{Y,Z}V @VVs_{Y\otimes X, Z}V \\
X\otimes Z\otimes Y  @>s_{X,Z\otimes Y}>> Z\otimes Y\otimes X
\end{CD}
$$ commutes.
\end{enumerate}
The collection of maps  $s_{X,Y}$ is called a commutor.

Let $X_1,\ldots,X_n\in Ob\ \CalC$. Then, according to \cite{HK}, the morphisms $X_1\otimes\ldots\otimes X_n\to X_{\sigma(1)}\otimes\ldots\otimes X_{\sigma(n)}$, for all possible $\sigma\in S_n$, which are compositions of some $s_{X,Y}$'s generate an action of $J_n$. Hence $PJ_n$ acts by endomorphisms of $X_1\otimes\ldots\otimes X_n$ for any collection of objects of any coboundary category.

\subsection{Crystal commutor.} Let $c(\l)=(\l,\l+2\rho)$ be the value of the Casimir operator on the irreducible $\fg$-module $V_\l$. It was noticed by Drinfeld that the morphism $\ol{R}:V_{\l_1}^q\otimes V_{\l_2}^q\to V_{\l_2}^q\otimes V_{\l_1}^q$ acting as $Rq^{\frac{-c(\l)+c(\l_1)+c(\l_2)}{2}}$ on $V_\l^q$-isotypic component of $V_{\l_1}^q\otimes V_{\l_2}^q$ is involutive and defines a structure of \emph{coboundary category} on $U_q(\fg)$-modules. The morphism $\ol{R}$ is called the \emph{unitarized} $R$-matrix.

One can define the unitarized $R$-matrix universally. Namely, there is an element $r$ in the completed tensor square of $U_q(\fg)$ acting as the composition of the flip and the operator $R$ in the tensor product of any pair of highest weight modules. Also, we have Drinfeld's Casimir operator $q^C$ in the completed $U_q(\fg)$ acting as $q^{c(\l)}$ in any $V^q_\l$. So, $\ol{R}$ can be written universally as $q^{\frac{C_2+C_2-C_{12}}{2}}\circ\text{flip}\circ r$.

Let $V^q_{\l_1}\otimes\ldots\otimes V^q_{\l_n}$ be a tensor product of irreducible $U_q(\fg)$-modules. The action of the cactus group $J_n$ on this space is defined as follows: the generator $s_{[k,l,m]}$ acts as the unitarized $R$-matrix $\ol{R_{[k,l];[l+1,m]}}$ transposing the neighboring factors $V^q_{\l_k}\otimes\ldots\otimes V^q_{\l_l}$ and $V^q_{\l_{l+1}}\otimes\ldots\otimes V^q_{\l_m}$.

It was shown in \cite{KT} that this morphism sends a crystal base to a crystal base hence giving a structure of a coboundary category on the category of $\fg$-crystals.

\begin{rem}  There are different definitions of the crystal commutor which work for any Kac-Moody $\fg$ and are equivalent to the definition above for finite dimensional $\fg$, see \cite{S}.
\end{rem}

\subsection{Example: $\fsl_2$ case.} For $\fg=\fsl_2$, the tensor product $\CB_{\l_1}\otimes\CB_{\l_2}$ is just the set $[0,\l_1]\times [0,\l_2]$, the element $(x,y)$ corresponds to the base vector $f^{(x)}v_{\l_1}\otimes f^{(y)}v_{\l_2}$. The commutor $s:[0,\l_1-1]\times [0,\l_2-1]\to[0,\l_1-1]\times [0,\l_2-1]$ is the following piecewise linear transformation:
\begin{equation}
s(x,y)=(y+(\l_1-x-y)_+-(\l_2-x-y)_+,x+(\l_2-x-y)_+-(\l_1-x-y)_+),
\end{equation}
where we set $(a)_+:=\max(0,a)$. This map is uniquely determined by the following property: each weight space of $\CB_{\l_1}\otimes\CB_{\l_2}$ is ordered by the values of the first coordinate $x$; the commutor $s$ preserves weight spaces and reverses the ordering on each of them.

Now let us describe the action of the generators $s_{p,q}\in J_n$ on $V^q_{\l_1}\otimes\ldots\otimes V^q_{\l_n}$. Choose a complete bracketing of the tensor product containing the pair of brackets bounding the segment $[p,q]$. Let $T$ be the corresponding binary rooted tree, $I$ be its inner vertex corresponding to this pair of brackets and $\ge$ be the partial order on inner vertices of $T$.

\begin{prop} The morphism $s_{p,q}:V_{\l_1}^q\otimes\ldots\otimes V_{\l_p}^q\otimes\ldots\otimes V_{\l_q}^q\otimes\ldots\otimes V_{\l_n}^q\to V_{\l_1}^q\otimes\ldots\otimes V_{\l_q}^q\otimes\ldots\otimes V_{\l_p}^q\otimes\ldots\otimes V_{\l_n}^q$ is given by the formula
$$
s_{p,q}=\prod\limits_{J\ge I}R_{L(J);R(J)}\cdot q^{-\frac{C_{LR(I)}}{2}}\prod\limits_{i\in LR(I)}q^{\frac{C_i}{2}}.
$$
\end{prop}

\begin{proof}
Straightforward from the definition of the unitarized $R$-matrix and the formula for the action of $s_{[k,l,m]}$.
\end{proof}

\section{Asymptotic solutions of the KZ equation}\label{sect-varchenko}

We reproduce here the results of Varchenko \cite{Var}. These results play the key role in the proof of the Etingof conjecture for $\fsl_2$.

Let ${\rm Conf}_n$ be the space of $n$-tuples of pairwise distinct complex numbers $(z_1,\ldots,z_n)$ such that $\sum\limits_{i=1}^n z_i=0$. Consider the Knizhnik-Zamolodchikov (KZ) connection on the trivial bundle on ${\rm Conf}_n$ with the fiber $\BV_{\ul{\l}}^{sing}$:
\begin{equation}
\nabla:=d-\frac{1}{\kappa}\sum\limits_{i=1}^n H_idz_i=d-\frac{1}{2\kappa}\sum\limits_{i<j}(C_{ij}-C_i-C_j)d\log(z_i-z_j).
\end{equation}

\begin{rem}
${\rm Conf}_n$ is the total space of the restriction of $\CL_{n+1}$ to the open stratum $M_{0,n+1}$. The Gaudin model can be regarded as a limit of the KZ connection as $\kappa\to0$.
\end{rem}

Let $V_\l^q$ be the irreducible $U_q(\fsl_2)$-module of highest weight $\l$. For $\ul{\l}=(\l_1,\ldots,\l_n)$, denote by $\BV_{\ul{\l}}^q$ the tensor product $V_{\l_1}^q\otimes\ldots\otimes V_{\l_n}^q$.

To any permutation $\sigma\in S_n$ we assign the connected component of ${\rm Conf}_n(\BR)$:
$$D_\sigma:=\{\ul{z}\in {\rm Conf}_n\ |\ z_{\sigma(i)}>z_{\sigma(i+1)}, i=1,\ldots,n-1\}.$$
For any planar binary rooted tree $T$ compatible with the ordering $\sigma$, we have the coordinates $u_I$ on $D_{\sigma}$. The functions $u_I$ give a diffeomorphism $D_{\sigma}\to\BR_{>0}^{n-1}$. The KZ equation in the coordinates $u_I$ has the following form:
\begin{equation}\label{KZ-asymptotic-zone}
\nabla=d-\frac{1}{2\kappa}\sum\limits_{I\in T}(C_{LR(I)}-\sum\limits_{i\in LR(I)}C_i)d\log u_I + \omega(u),
\end{equation}
where $\omega(u)$ is a $1$-form regular at $u_I=0$.

\begin{rem}The operators $C_{LR(I)}-C_{L(I)}-C_{R(I)}$ pairwise commute and have simple joint spectrum on $\BV_{\ul{\l}}^{sing}$. Hence they also have simple joint spectrum on any weight subspace $\BV_{\ul{\l}}(\mu)$ with respect to the diagonal $\fsl_2$-action.
\end{rem}

Let $v\in\BV_{\ul{\l}}(\mu)$ be a joint eigenvector of the operators $C_{LR(I)}-C_{L(I)}-C_{R(I)}$ with the eigenvalues $2\mu_I$. To such $v$ one assigns the \emph{asymptotic solution} to the equation~(\ref{KZ-asymptotic-zone}) on $D_\sigma$ of the form
\begin{equation}
\psi_{\sigma,T,v}(u)=\prod\limits_{I\in T} u_I^{\frac{\mu_I}{\kappa}}(v+\ol{o}(u)).
\end{equation}
We choose the branches of the functions $u_I^{\frac{\mu_I}{\kappa}}$ by the rule: $\arg(u_I^{\frac{\mu_I}{\kappa}}) = 0$ on $D_\sigma$. It is shown in \cite{Var} that $\psi_{\sigma,T,v}(u)$ has the following expansion with respect to $\kappa$:
\begin{equation}\label{asymptotic-solution}
\psi_{\sigma,T,v}(u)=\prod\limits_{I\in T} u_I^{\frac{\mu_I}{\kappa}}\exp(\frac{S(u)}{\kappa})\sum\limits_{j=0}^\infty f_{j,\sigma,T,v}(u)\kappa^j,
\end{equation}
where $S(u)$ is a real analytic function well-defined at $u_I=0$, and $f_{j,\sigma,T,v}(u)$ are $\BV_{\ul{\l}}$-valued real analytic functions well-defined at $u_I=0$. In particular $f_{0,\sigma,T,v}(u)$ is a joint eigenvector for $H_i(u)$ such that $f_{0,\sigma,T,v}(0)=v$. The functions~(\ref{asymptotic-solution}) form a basis of the space ${\rm Sol}_\sigma$ of solutions of the KZ equation on $D_\sigma$ (i.e. of the space of flat sections with respect to $\nabla$).

According to the Drinfeld-Kohno theorem, to each sector $D_\sigma\in {\rm Conf}_n(\BR)$ one can assign an isomorphism $\pi_\sigma$ between the space ${\rm Sol}_\sigma$ of solutions of the KZ equation on $D_\sigma$ and the space $\BV_{\sigma\ul{\l}}^q$ for $q=\exp(\frac{\pi i}{\kappa})$, such that the transitions between the chambers are given by the $R$-matrices. In \cite{Var} such isomorphisms $\pi_\sigma$ were explicitly constructed. More precisely, the collection of isomorphisms $\pi_\sigma$ make the following diagram commutative (here the upper horizontal arrow is the analytic continuation in counter-clockwise direction around the hyperplane $z_{\sigma(i)}=z_{\sigma(i+1)}$):
$$
\begin{CD}
Sol_\sigma @>>> Sol_{\ol{s_{i,i+1}}\sigma}   \\
@VV\pi_\sigma V @VV\pi_{\ol{s_{i,i+1}}\sigma}V \\
\BV_{\sigma\ul{\l}}^q  @> R_{i,i+1} >> \BV_{\ol{s_{i,i+1}}\sigma\ul{\l}}^q
\end{CD}
$$

The isomorphisms $\pi_{\sigma}$ from \cite{Var} have the following nice additional property which is crucial for Conjecture~\ref{conj-etingof}:

\begin{thm}\label{thm-varchenko}\cite{Var} There is a normalizing constant $N(\kappa, T, v)$ such that the image of the asymptotic solutions $N(\kappa, T, v)\psi_{\sigma,T,v}$ under $\pi_\sigma$ is a crystal base of $\BV_{\sigma\ul{\l}}^q$.
\end{thm}

\section{Proof of Etingof's conjecture for $\fsl_2$}\label{sect-main}

Let $\fg=\fsl_2$. According to Theorem~\ref{thm-varchenko}, each map $\pi_\sigma$ induces a bijection $\ol{\pi_\sigma}:B_{\ul{\l}}|_{D_\sigma}\to\CB_{\sigma\ul{\l}}$ which takes a Bethe eigenvector $f_{0,T,\sigma,v}(u)$ to the corresponding element of the crystal basis. Given an element $g\in J_n$, consider a path in $\ol{M_{0,n+1}}(\BR)$ representing it and connecting $D_\sigma$ with $D_{\ol{g}\sigma}$. Each element $f\in B_{\ul{\l}}|_{D_\sigma}$ can be analytically continued along this path thus giving an element $g_{an}(f)\in B_{\ul{\l}}|_{D_{\ol{g}\sigma}}$. On the other hand, for any $b\in\CB_{\sigma\ul{\l}}$ there is an element $g(b)\in\CB_{\ol{g}\sigma\ul{\l}}$.

\begin{thm}\label{main-theorem}\footnote{This result together with the idea of the proof was suggested by Pavel Etingof.} For any $g\in J_n$, the following diagram commutes.
$$
\begin{CD}
B_{\ul{\l}}|_{D_\sigma} @> g_{an}>> B_{\ul{\l}}|_{D_{\ol{g}\sigma}}   \\
@VV\ol{\pi_\sigma}V @VV\ol{\pi_{\ol{g}\sigma}}V \\
\CB_{\sigma\ul{\l}}  @> g >> \CB_{\ol{g}\sigma\ul{\l}}
\end{CD}
$$
\end{thm}

\begin{proof} Let $\sigma'=s_{[k,l,m]}\sigma$, $T$ be any tree compatible with $s_{[k,l,m]}$ and $I$ be the corresponding inner vertex of $T$. Let $U_{\sigma,T}$ be the corresponding chart. From Theorem~\ref{thm-varchenko}, we see that, for each asymptotic solution $\psi_{\sigma,T,v}$ on $D_\sigma$, the $\BV_{\ul{\l}}$-valued function $f_{0,\sigma,T,v}$ on $D_\sigma$ continues to $U_{\sigma,T}$ as a well-defined real analytic function. Clearly, its restriction to $D_{\sigma'}\subset U_{\sigma,T}$ is $f_{0,\sigma',T,v}$. To prove the Theorem, it suffices to show that $\pi_{\sigma'}(\psi_{\sigma',T,v})=\ol{R}_{[k,l,m]}\pi_{\sigma}(\psi_{\sigma,T,v})$.

The asymptotic solution $\psi_{\sigma,T,v}(u)$ continues as a holomorphic function (in counter-clockwise direction) to $U_{\sigma',T}$, and by our choice of the branch of $u_I^{\frac{\mu_I}{\kappa}}$ we have $\pi_{\sigma'}(\psi_{\sigma',T,v})=\pi_{\sigma'}(q^{-\mu_I}\psi_{\sigma,T,v})=\pi_{\sigma'}(q^{-\frac{1}{2}(C_{LR(I)}-C_{L(I)}-C_{R(I)})}\psi_{\sigma,T,v})=
q^{-\frac{1}{2}(C_{LR(I)}-C_{L(I)}-C_{R(I)})}\pi_{\sigma'}(\psi_{\sigma,T,v})$. By Drinfeld-Kohno theorem, the latter is equal to
$q^{-\frac{1}{2}(C_{LR(I)}-C_{L(I)}-C_{R(I)})}R_{[k,l,m]}\pi_{\sigma}(\psi_{\sigma,T,v})=
\ol{R}_{[k,l,m]}\pi_{\sigma}(\psi_{\sigma,T,v})$.
\end{proof}

\begin{cor} There is a bijection $B_{\ul{\l}}(\ul{z})\to\CB_{\ul{\l}}$ commuting with the action of $PJ_n$.
\end{cor}

\section{Piecewise linear transformations.}\label{sect-piecelinear}
We present here a more elementary proof of Conjecture~\ref{conj-etingof} for $\fg=\fsl_2$ using the description of the coboundary category of $\fsl_N$-crystals from \cite{HK2}. More precisely, in \cite{HK2} Henriques and Kamnitzer define (in a purely combinatorial way) some different coboundary category $\rm{HIVES}$ where the associator and commutor are \emph{both nontrivial} and prove that it is equivalent to the category of crystals. The general definition of $\rm{HIVES}$ is complicated, but for $\fsl_2$ it simplifies and gives the following. $\rm{HIVES}$ is a semisimple category whose simple objects $L(\lambda)$ are indexed by nonnegative integers $\lambda\in\BZ_{\ge0}$. The tensor product $L(\l_1)\otimes L(\l_2)$ is the union of $L(\mu)$ where $|\l_1-\l_2|\le\mu\le\l_1+\l_2$ and $\l_1+\l_2-\mu\in2\BZ$. The occurrences of $L(\nu)$ in the triple tensor product $(L(\l_1)\otimes L(\l_2))\otimes L(\l_3)$ are thus indexed by the set $M_{(\l_1\l_2)\l_3}^\nu:=\{\mu\ |\ \max(|\l_1-\l_2|,|\nu-\l_3)\le\mu\le\min(\l_1+\l_2,\nu+\l_3)\}$. The occurrences of $L(\nu)$ in the same triple tensor product with another bracketing $L(\l_1)\otimes (L(\l_2)\otimes L(\l_3))$ are indexed by the set $M_{\l_1(\l_2\l_3)}^\nu:=\{\mu\ |\ \max(|\l_3-\l_2|,|\nu-\l_1|)\le\mu\le\min(\l_3+\l_2,\nu+\l_1)\}$. The associator (associativity morphism) $\psi: L(\l_1)\otimes L(\l_2))\otimes L(\l_3)\to L(\l_1)\otimes (L(\l_2)\otimes L(\l_3)$ is given by the map
$$ \psi: M_{(\l_1\l_2)\l_3}^\nu\to M_{\l_1(\l_2\l_3)}^\nu,\quad \mu\mapsto \max(\l_1+\l_3,\l_2+\nu)-\mu.
$$
The commutor (commutativity morphism) $s:L(\l_1)\otimes L(\l_2)\to L(\l_2)\otimes L(\l_1)$ is given by the identity map on the set of occurrences of each $L(\mu)$ (which is either empty or $1$-element).

\begin{thm}\cite{HK2} The category of $\fsl_2$-crystals is equivalent to $\rm{HIVES}$.
\end{thm}

Let $v$ be any ordered bracketing of the tensor product of $L(\l_i)$. Then the set $M_{v(\l_1,\ldots,\l_n)}^\nu$ indexing the occurrences of $L(\nu)$ in the tensor product of $L(\l_i)$ according to the ordered bracketing $v$ is the set of integer points of a \emph{convex polytope}. These polytopes are different for different $v$, but always have the same number of integer points. Note that the polytopes depend only on the equivalence class of an ordered bracketing, and hence we can regard these polytopes as attached to the vertices (i.e. $0$-dimensional strata) of $\ol{M_{0,n+1}}(\BR)$. The associators and commutors act by some piecewise linear transformations between the polytopes attached to neighboring vertices of $\ol{M_{0,n+1}}(\BR)$. In particular, for $n=3$, the sets $M_{(\l_1\l_2)\l_3}^\nu$ and $M_{\l_1(\l_2\l_3)}^\nu$ are both segments of the length $\min(\l_3+\l_2,\nu+\l_1)-\max(|\l_3-\l_2|,|\nu-\l_1|)$. The increasing order on the highest weights $\mu$ defines an orientation on these segments. The associator $\psi: M_{(\l_1\l_2)\l_3}^\nu\to M_{\l_1(\l_2\l_3)}^\nu$ reverses the (increasing) order on the set of integer points of the segments. Clearly, $\psi$ is uniquely determined by this property.

Now let us see the same structure from the Gaudin algebras acting on the tensor product $\BV_{\ul{\l}}=V_{\lambda_1}\otimes\ldots\otimes V_{\lambda_n}$. To each ordered bracketing of the tensor product of irreducible finite dimensional $\fsl_2$-modules $V_{\lambda_1}\otimes\ldots\otimes V_{\lambda_n}$ we assign a basis of the space $\BV_{\ul{\l}}^{sing}$ obtained by iterating the decomposition of two irreducible $\fsl_2$-modules according to the bracketing. This basis consists of joint eigenvectors for the operators $C_J\in U(\fsl_2)^{\otimes n}$ where $J\subset [1,n]$ is the set of indices inside a pair of brackets for all pairs of brackets. Note that the $0$-dimensional strata of $\overline{M_{0,n+1}}(\BR)$ correspond to complete ordered bracketings $v$ of the set $\{\l_1,\ldots,\l_n\}$ up to transpositions of factors inside any pair of brackets, and the algebra generated by these Casimirs is the Gaudin algebra corresponding to this stratum. Thus the eigenbasis for the Gaudin algebra $\A(\ul{z}_v)$ is naturally indexed by $M_{v(\l_1,\ldots,\l_n)}^\nu$.

\subsection{Important example.} Let $n=3$, then $\ol{M_{0,4}}(\BR)=\BR\BP^1$. We define the coordinate on $\BR\BP^1$ by $t=\frac{z_1-z_2}{z_1-z_3}$, then the $0$-dimensional strata are the points $0,1,\infty$. Each of these points correspond to some equivalence classes of ordered bracketing of $\l_1\l_2\l_3$ in the following way:
$$
0 \to (\l_1\l_2)\l_3=(\l_2\l_1)\l_3=\l_3(\l_1\l_2)=\l_3(\l_2\l_1);
$$
$$
1 \to \l_1(\l_2\l_3)=\l_1(\l_3\l_2)=(\l_2\l_3)\l_1=(\l_3\l_2)\l_1;
$$
$$
\infty \to (\l_1\l_3)\l_2=(\l_3\l_1)\l_2=\l_2(\l_1\l_3)=\l_2(\l_3\l_1).
$$

The basis of the $\nu$-weight subspace of $\BV_{\ul{\l}}^{sing}$ corresponding to the point $0$ is indexed by the highest weights $\mu$ such that $\max(|\l_1-\l_2|,|\nu-\l_3)\le\mu\le\min(\l_1+\l_2,\nu+\l_3)$ and $\max(|\l_1-\l_2|,|\nu-\l_3)-\mu$ is even. The eigenvalue of $C_{12}$, the only nontrivial generator of the corresponding Gaudin algebra, on such vector is $\frac{\mu(\mu+2)}{2}$. In particular, this eigenvalue is an increasing function of $\mu$. The same is true for other $0$-dimensional strata.

\begin{prop} The transport from the point $0\in \ol{M_{0,4}}(\BR)=\BR\BP^1$ to the point $1\in \ol{M_{0,4}}(\BR)=\BR\BP^1$ along the interval $(0,1)$ acts as $\psi: M_{(\l_1\l_2)\l_3}^\nu\to M_{\l_1(\l_2\l_3)}^\nu$.
\end{prop}

\begin{proof} The Bethe basis at each point $t\in[0,1]$ is the eigenbasis for the operator $H(t):=(1-t)C_{12}-tC_{23}$. For any $t$ this operator has pairwise distinct real eigenvalues on $\BV_{\ul{\l}}^{sing}$, hence the basis is determined by $H(t)$, and moving $t$ along the segment $[0,1]$ preserves the order of the eigenvalues of $H(t)$. We have $H(0)=C_{12}$ and $H(1)=-C_{23}$. Hence the transport along the segment takes the spectrum of $C_{12}$ in the \emph{increasing order} to the spectrum of $C_{23}$ in the \emph{decreasing order}. Thus the transport along $[0,1]$ acts as the associator in the category of hives.
\end{proof}

\begin{cor} The transports along $1$-dimensional strata of $\overline{M_{0,n+1}}(\BR)$ act as associators from \cite{HK2}.
\end{cor}

\begin{cor} The Conjecture~\ref{conj-etingof} is valid for $\fsl_2$.
\end{cor}

\section{Discussion.}\label{sect-conjectures}

\subsection{Bethe Ansatz conjecture.} The following statement is a variant of Bethe Ansatz conjecture.

\begin{conj} Theorem~\ref{thm-cyclic} holds for arbitrary $\fg$.
\end{conj}

In particular, this means that the spectrum of the Gaudin algebra $\A(\ul{z})$ is simple for any $\ul{z}\in\ol{M_{0,n+1}}(\BR)$. Then for any collection of highest weights $\l_1,\ldots,\l_n$ we have a finite covering of the Deligne-Mumford moduli space $\ol{M_{0,n+1}}(\BR)$ whose fiber is the spectrum of the corresponding commutative algebra in space of highest vectors of the tensor product of irreducible $\fg$-modules $V_{\l_1},\ldots ,V_{\l_n}$.

\subsection{Opers and crystals.}

According to Feigin and Frenkel, the spectrum of the Gaudin model is (modulo Bethe Ansatz conjecture) in 1-1 correspondence with the set of monodromy-free ${}^L G$-opers on $\BP^1$ with regular singularities of type $\l_i$ at the marked points $z_i$ and a regular singularity at $\infty$. One can define a monodromy-free oper on a nodal curve as a collection of monodromy-free opers on on each component with regular singularities an the marked points and at the nodes such that for any pair of intersecting components the corresponding opers have the same type of singularity at the intersection point. Generalizing the second proof of Theorem~\ref{main-theorem}, one can define a coboundary monoidal category $\rm{OPERS}$ whose simple objects $L_\lambda$ are indexed by the set of dominant integral weights $\lambda$ of $\fg$, and the tensor product is defined by the following rule: $L_\l\otimes L_\mu=\bigoplus\limits_\nu M_{\l,\mu,\nu}\times L_{\nu^*}$ where $M_{\l,\mu,\nu}$ is the set of monodromy-free ${}^L G$-opers on $\BC\BP^1$ having regular singularities with the residues $\l,\mu,\nu$ at the points $0,1,\infty$, respectively, and regular at other points (for a dominant integral weight $\nu\in X$ define the \emph{dual} weight $\nu^*\in X$ by the property $V_\nu^*=V_{\nu^*}$). The set $\bigcup\limits_{\nu}M_{\l,\mu,\nu}\times M_{\nu,\delta,\epsilon}$ can be regarded as the space of monodromy-free opers on the degenerate stable rational curve with $4$ marked points with regular singularities of residues $\l,\mu,\delta,\epsilon$ at the marked points. One can define a transport of the set of opers along the shortest path in $\ol{M_{0,4}}(\BR)$ connecting two degenerate curves:
$$\psi: \bigcup\limits_{\nu}M_{\l,\mu,\nu}\times M_{\nu,\delta,\epsilon}\to \bigcup\limits_{\nu}M_{\l,\nu,\epsilon}\times M_{\mu,\delta,\nu}.$$

We can also define a bijection
$$ s: M_{\l,\mu,\nu}\to M_{\mu,\l,\nu}
$$
as the map of the set of opers induced by the holomorphic automorphism $z\mapsto (1-z)$ of $\BC\BP^1$.

\begin{conj} The above category with the maps $\psi$ as the associator morphisms and $s:M_{\l,\mu,\nu}\to M_{\mu,\l,\nu}$ as the commutor morphisms is a coboundary monoidal category. Moreover, it is equivalent to the category of $\fg$-crystals.
\end{conj}

\begin{cor} In particular, for $\fg=\fsl_N$, we get a bijection between Bethe vectors in the space of invariants in the triple tensor product and the corresponding set of hives.
\end{cor}

\subsection{Shift of argument subalgebras.} There should be an analog of Theorem~\ref{main-theorem} for shift-of-argument subalgebra for arbitrary $\fg$, cf. \cite{Ryb1,FFR10}. Namely, there is a family of maximal commutative subalgebras $\A_\mu\subset U(\fg)$ parameterized by regular elements $\mu\in\fh$ (in fact, the subalgebra $\A_\mu$ does not change under the dilations of $\fh$). The space $\BP(\fh_{reg})$ parameterizing the family $\A_\mu$ is noncompact. On the other hand, each subalgebra from this family is a point in the appropriate Grassmannian which is compact. Hence there is a natural compactification of $\BP(\fh_{reg})$ parameterizing some commutative subalgebras of $U(\fg)$ which have the same Poincar\'e series as $\A_\mu$.

\begin{conj} This compactification is the De Concini-Procesi wonderful compactification for the root hyperplane arrangement, see \cite{DCP}.
\end{conj}

The set-theoretical version of this conjecture was proved by V.Shuvalov in 2002, see \cite{Sh}.

\begin{conj} The algebras corresponding to the real points of the above compactification act with simple spectrum on any finite-dimensional irreducible representation of $\fg$.
\end{conj}

The conjecture is proved in type A (in fact one can deduce this from the same fact for Gaudin algebras by Howe duality). Also, it is proved in \cite{FFR10} that for real points $\mu\in \fh_{reg}$ the corresponding algebra $\A_\mu$ has simple spectrum in any irreducible $\fg$-module $V_\l$. So it remains to show that this also holds for boundary points. This gives an action of the fundamental group of the De Concini-Procesi wonderful compactification on the set of eigenvectors of $\A_\mu$ in any irreducible finite-dimensional representation $V_\l$. On the other hand, there is an action of the same fundamental group on the corresponding crystal, which can be obtained from Lusztig's braid group action on the irreducible representation of $U_q(\fg)$ on the irreducible representation $V_\l$ by Drinfeld's unitarization procedure and taking the limit $q\to\infty$.

\begin{conj}\footnote{This was independently conjectured by Joel Kamnitzer and Alex Weekes.} The above two actions of the fundamental group are isomorphic.
\end{conj}

\subsection{Relation to the results of D. Speyer and K.Purbhoo.} In \cite{Sp} Speyer defines a covering of $\ol{M_{0,n+1}}(\BR)$ whose fiber at a generic point is an intersection of some Schubert varieties in certain Grassmannian.  On the other hand, Mukhin, Tarasov and Varchenko \cite{MTV07} prove that there is a 1-1 correspondence between the Bethe vectors of the $GL_n$ Gaudin model and the intersection of the same Schubert varieties in the same Grassmannian. So, it is natural to expect that the following is true:
\begin{conj}\label{conj-speyer}
The covering of $\ol{M_{0,n+1}}(\BR)$ from \cite{Sp} is isomorphic to our covering of $\ol{M_{0,n+1}}(\BR)$ from Corollary \ref{cor-covering}.
\end{conj}
Theorem 1.6 of \cite{Sp} shows that the combinatorics of this covering can be described in terms of $GL_N$-crystals, so the analog of Etingof's conjecture for this covering is mostly proved in \cite{Sp}. In particular, we expect that Conjecture~\ref{conj-speyer} implies Conjecture~\ref{conj-etingof}.

In \cite{Purb}, Purbhoo studies the monodromy problem for Wronskians, and the answer is given in terms of Jeu de taquin (which is the same as crystal commutor for tautological representations of $GL_N$, due to \cite{HK}). On the other hand, the results of Mukhin, Tarasov and Varchenko imply that the monodromy problem for Wronskians is the same as the monodromy problem for Bethe eigenvalues for $GL_N$ Gaudin model with tautological representations of $GL_N$. So we expect that the results of \cite{Purb} imply Conjecture~\ref{conj-etingof} for tautological representations of $GL_N$.

\subsection{Berenstein-Kirillov group.} Let $\fg=\fgl_N$ and $V_{\l_i}$ be the symmetric powers of the standard representation $V=\BC^N$. By the $GL_N-GL_n$ Howe duality one describes the spectrum of the algebra $\A(\ul{z})$ at the \emph{caterpillar point} $\ul{z}$ of $\ol{M_{0,n+1}}$ (i.e. $\ul{z}$ corresponding to the stable rational curve having exactly $3$ distinguished points and at least $1$ marked point on each component) as the set of Gelfand-Tsetlin tables for the group $GL_n$. The latter is the set of integral points of the Gelfand-Tsetlin convex polytope. So we have an action of the pure cactus group on the Gelfand-Tsetlin polytope. On the other hand, Berenstein and Kirillov described in \cite{BK} some group generated by involutions acting on the Gelfand-Tsetlin polytope by piecewise linear transformations.

\begin{conj}
Berenstein-Kirillov group is a quotient of the pure cactus group $PJ_n$.
\end{conj}

We checked this for Gelfand-Tsetlin polytopes corresponding to 2-row Young diagrams. It turns out that the involutions generating Berenstein-Kirillov group come from the loops around the $\BR\BP^1$'s embedded as $\ol{M_{0,4}}(\BR)\subset \ol{M_{0,n+1}}(\BR)$.

\end{document}